\documentclass{amsart}
\usepackage{amssymb}
\usepackage{amsmath}

\newtheorem{theo}{Theorem}[section]
\newtheorem{lemm}[theo]{Lemma}
\newtheorem{defi}[theo]{Definition}
\newtheorem{coro}[theo]{Corollary}
\newtheorem{prop}[theo]{Proposition}
\newtheorem{rema}[theo]{Remark}

\def\la{\langle}
\def\ra{\rangle}

\newcommand\NN{{\mathbb N}}
\newcommand\RR{{{\mathbb R}}}

\def\SS {\mathbb{S}}
\newcommand\cA{{\mathcal A}}

\newcommand\cC{{\mathcal C}}
\newcommand\cD{\mathcal D}

\newcommand\cS{{\mathcal S}}

\begin{document}
\title[Smoothing effect of  weak solutions]
{Smoothing effect of  weak solutions\\ for the spatially  homogeneous\\
Boltzmann Equation without angular cutoff}
\author{R. Alexandre}
\address{Radjesvarane Alexandre
\newline\indent
Department of Mathematics, Shanghai Jiao Tong University
\newline\indent
Shanghai, 200240, P. R. China\newline\indent
and
\newline\indent
IRENAV Research Institute, French Naval Academy
Brest-Lanv\'eoc 29290, France}
\email{radjesvarane.alexandre@ecole-navale.fr}
\author{Y. Morimoto }
\address{Yoshinori Morimoto
\newline\indent
Graduate School of Human and Environmental Studies,
Kyoto University
\newline\indent
Kyoto, 606-8501, Japan} \email{morimoto@math.h.kyoto-u.ac.jp}
\author{S. Ukai}
\address{Seiji. Ukai,
\newline\indent
17-26 Iwasaki-cho, Hodogaya-ku, Yokohama 240-0015, Japan}
\email{ukai@kurims.kyoto-u.ac.jp}
\author{C.-J. Xu}
\address{Chao-Jiang Xu
\newline\indent
School of Mathematics, Wuhan University 430072,
Wuhan, P. R. China
\newline\indent
and \newline\indent
Universit\'e de Rouen, UMR 6085-CNRS,
Math\'ematiques
\newline\indent
Avenue de l'Universit\'e,\,\, BP.12, 76801 Saint
Etienne du Rouvray, France } \email{Chao-Jiang.Xu@univ-rouen.fr}
\author{T. Yang}
\address{Tong Yang
\newline\indent
Department of mathematics, City University of Hong Kong,
Hong Kong, P. R. China
\newline\indent
and \newline\indent
School of Mathematics, Wuhan University 430072,
Wuhan, P. R. China} \email{matyang@cityu.edu.hk}

\subjclass[2000]{35A05, 35B65, 35D10, 35H20, 76P05, 84C40}

\keywords{Boltzmann equation, weak solution, smoothing effect.}

\date{29-April-2011}

\begin{abstract}
In this paper, we consider the spatially homogeneous Boltzmann equation without angular cutoff. We prove that every $L^1$ weak solution to the Cauchy problem with finite moments of all order acquires the $C^\infty$ regularity in the velocity variable for the positive time.
\end{abstract}

\maketitle

\section{Introduction}\label{s1}
Consider the Cauchy problem for the spatially homogeneous Boltzmann
equation,
\begin{equation}\label{1.1}
\left\{
\begin{array}{l}\displaystyle
f_t(t,v)=Q(f,f)(t,v),\quad t\in \RR^+, \,  v \in \RR^3,
\\
f(0,v)=f_0(v), \end{array}
\right.
\end{equation}
where $f= f(t,v)$ is the density distribution function
of particles  with velocity $v\in \RR^3$ at time $t$.
The right hand side of (\ref{1.1}) is given by the
Boltzmann bilinear collision operator
\[
Q(g, f)=\int_{\RR^3}\int_{\mathbb S^{2}}B\left({v-v_*},\sigma
\right)
 \left\{g(v'_*) f(v')-g(v_*)f(v)\right\}d\sigma dv_*\,,
\]
which is well-defined for
 suitable functions $f$ and $g$ specified later. Notice that the
collision operator $Q(\cdot\,,\,\cdot)$ acts only on the velocity
variable $v\in\RR^3$. In the following discussion, we will use the
$\sigma-$representation, that is, for $\sigma\in \mathbb S^{2}$,
$$
v'=\frac{v+v_*}{2}+\frac{|v-v_*|}{2}\sigma,\,\,\, v'_*
=\frac{v+v_*}{2}-\frac{|v-v_*|}{2}\sigma,\,
$$
which give the relations between the post and
pre collisional velocities.
For mono-atomic gas,  the non-negative cross section
 $B(z, \sigma)$ depends only on $|z|$ and the scalar product
$\frac{z}{|z|}\,\cdot\, \sigma$. As in \cite{amuxy4-1, amuxy4-2, amuxy4-3}, we assume
that it takes  the form
\begin{equation}\label{part2-hyp-2}
B(v-v_*, \cos \theta)=\Phi (|v-v_*|) b(\cos \theta),\,\,\,\,\,
\cos \theta=\frac{v-v_*}{|v-v_*|} \, \cdot\,\sigma\, , \,\,\,
0\leq\theta\leq\frac{\pi}{2},
\end{equation}
in which it contains a kinetic factor given by
\begin{equation}
\label{part2-hyp-phys}
\Phi (|v-v_*|) =\Phi_\gamma(|v-v_*|)= |v-v_*|^{\gamma},
\end{equation}
with $\gamma >-3$ and a factor related to the collision angle with singularity,
\begin{equation}\label{angular}
\begin{array}{l}
 b(\cos \theta) \theta^{2+2s}\rightarrow K, \
\
 \mbox{when} \ \ \theta\rightarrow 0+,
\end{array}
\end{equation}
for some positive constant $K$ and $0< s <1$.

The main purpose of this paper is to show the smoothing effect of the spatially homogeneous
Boltzmann equation, that is,  any
weak solution to the Cauchy problem \eqref{1.1}
acquires  regularity as soon as $t >0$.
Let us recall the precise definition of weak solution for the Cauchy problem \eqref{1.1} given in  \cite{villani}, see also \cite{villani2}. To this end,
we introduce  the standard notation,
\begin{align*}
&\|f\|_{L^p_\ell}=\left(\int_{\RR^3}|f(v)|^p(1+|v|)^{\ell p} d v\right)^{1/p}, \enskip \mbox{for} \enskip p \ge 1 , \ell \in \RR,\,\\
&\|f\|_{L\log L}=\int_{\RR^3}
|f(v)|\log (1+|f(v)|)d v\,.
\end{align*}
\begin{defi}\label{defi1}
Let $f_0\ge 0$ be a function defined on $\RR^3$ with finite mass,
energy and entropy, that is,
\begin{equation*}
\displaystyle\int_{\RR^3}f_0(v)[1+|v|^2+ \log(1+f_0(v))]dv<+\infty.
\end{equation*}
$f$ is  a weak solution of the Cauchy problem \eqref{1.1}, if it satisfies the following conditions:
\begin{align*}
&f \ge 0, \ f \in C(\RR^+; \cD'(\RR^3))\cap L^1([0, T];
L^1_{2+\gamma^+}(\RR^3)),\\
& f(0,\, \cdot\, )=f_0(\, \cdot\,), \\
& \int_{\RR^3} f(t, v)\psi(v) dv=\int_{\RR^3} f_0(v)\psi(v) d v \
\, \mbox{for}\,\,\psi=1, v_1, v_2, v_3, | v|^2;
\\
& f(t,\,\cdot)\in L \log L,\ \ \int_{\RR^3} f(t,v) \log
f(t,v)dv\leq \int_{\RR^3} f_0\log f_0dv, \ \
\forall t\geq 0;  \\
&\int_{\RR^3} f(t,v)\varphi(t,v)dv-\int_{\RR^3} f_0(v)\varphi(0,v)dv
-\int^t_0 d\tau \int_{\RR^3} f(\tau,v)\partial_{\tau} \varphi(\tau,v)dv \\
&\hspace{4cm} =\int^t_0 d\tau \int_{\RR^3} Q(f,
f)(\tau,v)\varphi(\tau,v)dv,&
\end{align*}
where $\varphi\in C^1(\RR^+; C^\infty_0(\RR^3))$. Here, the
right hand side of the last integral given above is defined by
\begin{align*}
&\int_{\RR^3} Q(f, f)(v)\varphi(v) d v\\
&=\frac{1}{2}\int_{\RR^6}\!\!\int_{\mathbb S^2}\!\! B\,
f(v_*)f(v)(\varphi(v')+\varphi(v'_*)
-\varphi(v)-\varphi(v_*))dv dv_*
d\sigma.
\end{align*}
Hence, this integral is well defined for any test function
$\varphi\in L^{\infty}([0, T]; W^{2, \infty}(\RR^3))$ (see p. 291 of \cite{villani}).
\end{defi}

To state the main theorem in this paper, we introduce
the entropy dissipation functional by
\[
D(g,f) = -\iiint_{\RR^3\times \RR^3\times \SS^2} B\, \big (g'_* f' -g_*  f  \big )\log f dv dv_* d\sigma\,,
\]
where $f=f(v), f'=f(v'), g_* = g(v_*), g'_* = g(v'_*)$.

\begin{theo}\label{Theorem1}
Let the cross section $B$ in the form \eqref{part2-hyp-2} satisfy  \eqref{part2-hyp-phys} and \eqref{angular}
with $0<s<1$.

\noindent
{\em 1) }  Suppose that  $\gamma > \max\{ -2s, -1\} $.
Let  $f$ be a weak solution of the Cauchy problem \eqref{1.1}.  For
$0\le T_0<T_1$, if $f$ satisfies
\begin{equation}\label{moment}
\mbox{$|v|^\ell f \in L^\infty([T_0,T_1]; $ $L^1(\RR^3))$ \enskip for any \enskip $\ell\in\NN$},
\end{equation}
then
\begin{align*}
f\in L^\infty([t_0,T_1]; \ \mathcal{S}(\RR^3)),
\end{align*}
for any $t_0\in ]T_0, T_1[\,\,.$

\noindent
{\em 2) }  When   $-1 \ge \gamma > -2s$,  the same conclusion as above holds
 if we have the following entropy dissipation estimate
\begin{align}\label{entropy-dissipation-estimate-0}
 \int_{T_0}^{T_1}  D(f(t),f (t)) dt < \infty\,.
\end{align}
\end{theo}

The existence of
weak solutions to the  Cauchy problem \eqref{1.1}
was proved by Villani \cite{villani} when  $\gamma \ge -2$, assuming additionally in the case $\gamma >0$
 that $f_0 \in L^1_{2+\delta}$ for some $\delta >0$.
One important property of the weak solution for the
hard potentials (namely when $\gamma >0$) is, according to the work by Wennberg \cite{We96} (cf. also Bobylev\cite{bobylev-2}),
 the moment gain property. It means that $f$ satisfies \eqref{moment} for arbitrary $T_0>0$
  when the initial data
only satisfies  finite mass, energy and entropy.
However, without assuming the moment condition \eqref{moment},
we can still consider the smoothing effect in case with mild singularity ($0 < s< 1/2$) for the hard potential($\gamma>0$),
and the argument is
similar to the one used in \cite{HMUY} ( see Theorem \ref{H-infty-smooth} in Section
\ref{section5}).

This kind of regularization
property has been studied by many authors, cf. \cite{alexandre1,al-saf-1,desv-wen1,HMUY,MUXY-DCDS,ukai}. However, to our knowledge, it
has not yet been completely established in the sense that
the kinetic factor $\Phi(|z|)$ was modified
to avoid the singularity at the origin except the Maxwellian molecule case in previous works, and moreover some extra
conditions other than those in  Definition 1.1 of weak solution were required in \cite{al-saf-1,desv-wen1}.

We would like to emphasize that the result of Theorem \ref{Theorem1} gives the full regularization
property for any weak solution satisfying some natural boundedness condition
in some weighted $L^1$ and $L\log L$ space, that requires no differentiation
on the solution.

Recently in \cite{H-C}, it was proved that  $W^{1, 1}_p\cap H^3$ (strong) solutions
gain full regularity in the case $0<s<1/2$. Their method is based on the a priori estimate of the smooth solution, together with results given in \cite{dev-mouhot} about
the propagation of the norm $W^{1, 1}_p$ and the uniqueness of the solution.
Different from \cite{H-C}, we
start from the weak
solution given in Definition \ref{defi1} without any known uniqueness
result. Therefore, a priori estimate for the smooth function is not enough to show the
regularity for the weak solution in $L^1$ with moments. For the proof of Theorem \ref{Theorem1},
some suitable mollifier, acting to the weak solution,
becomes necessary, so that its commutator with
the collision operator requires some subtle analysis.

Throughout this paper, we will use the following notations:
$f\lesssim g$ means that there exists a generic positive constant C
such that $f\le Cg$; while $f\gtrsim g$ means $ f\ge Cg$. And $f\sim
g$ means that there exist two  generic positive constant $c_1$ and
$c_2$ such that $c_1 f\le g\le c_2 g$.

The rest of the paper will be organized as follows. In the next section,
we will prove a uniform coercivity estimate that improves the one given
in \cite{advw} which has its own interest. The mollifier and the commutator estimate will be
given in Section \ref{section2}. In Section \ref{section4} we prove the smoothing effect of weak
solution with extra $L^2$ assumption.
The last section is devoted to the proof of Theorem \ref{Theorem1}.

\section{A uniform coercive estimate}\label{sect-IV-10}
\smallskip \setcounter{equation}{0}

In this section, we will improve the coercive estimate for the collision operator
obtained in \cite{advw} by removing the restriction on $v$ in a bounded domain.

In view of the definition of the weak solution,  for $D_0 , E_0 >0$ we set
\[
{\mathcal{U}}(D_0,E_0)=
\{ g \in L^1_2 \mathop{\cap} L\log L\,\, ; \,\,
g \ge 0\,, \enskip \|g\|_{L^1}\ge D_0, \enskip \, \|g\|_{L^1_2} + \|g \|_{L\log L} \le E_0\,\,\}\,.
\]
Set  $B(R) = \{v \in \RR^3\,;\, |v| \le R\}$ for $R >0$ and set  $B_0(R,r) =
\{ v \in B(R)\,; \, |v-v_0| \ge r\}$ for  a $v_0 \in \RR^3$ and $r \ge 0$.
It follows from  the definition of ${\mathcal{U}}(D_0,E_0)$ that there
exist positive constants $R > 1 > r_0$ depending only on $D_0, E_0$ such that
\begin{equation}\label{uni-g}
\mbox{
$g \in {\mathcal{U}}(D_0,E_0)$\,\, \mbox{implies} \,\,$\chi_{B_0(R, r_0)} g \in {\mathcal{U}}(D_0/2,E_0)$
}\,,
\end{equation}
where $\chi_A$ denotes a characteristic function of the set $A \subset \RR^3$.
In fact, noting that for $R, M> 0$
\[R^2 \int_{\{|v|> R\}} g dv  + \log (1+M) \int_{\{ g >M\}} g dv \le
E_0\, .
\]
We have
$$\int_{\{|v| \le R \}\cap\{g \le M\} } g dv \ge 3D_0/4
$$
if $R \ge 2 \sqrt{2E_0/D_0}$ and $\log(1+M) \ge 8E_0/D_0$, moreover
we have
$$
\int_{\{|v-v_0| < r_0  \}\cap\{g \le M\} } g dv \le D_0/4
$$
if $r_0 \le (3D_0/(16\pi \exp(8E_0/D_0))^{1/3}$.
\begin{prop}\label{coercive-uniform}
Suppose that the cross section $B$ of the form \eqref{part2-hyp-2} satisfies  \eqref{part2-hyp-phys} and \eqref{angular}
with $0<s<1$ and  $\gamma >-3$. If $D_0, E_0>0$
and if   $g \in {\mathcal{U}}(D_0,E_0)$ then there exist positive constants  $c_0, C $ depending only on
$D_0, E_0$ such that for any $f \in {\mathcal S}(\RR^3)$,
\begin{align}\label{coer-uni}
 -\Big(
Q(g,f)\,,\,f\Big)_{L^2} \ge c_0 \|\la v \ra^{\gamma/2} f\|_{H^s}^2 - C
 \|\la v \ra^{\gamma/2}f\|^2_{H^{(-\gamma/2)^+}},
\end{align}
where $a^+ = \max\{ a, 0\}$ for $a \in \RR$. Furthermore,  if $\gamma+ 2s \le 0$, $0< s'<s$ and
if $g$ belongs to $L^{3/(3+\gamma+2s')}_{-\gamma}$  then there exists a $C_1 >0$ independent of $g$ such that
for any $f \in {\mathcal S}(\RR^3)$,
\begin{align}\label{coer-uni-app}
 -\Big(
Q(g,f)\,,\,f\Big)_{L^2} \ge c_0 \|\la v \ra^{\gamma/2} f\|_{H^s}^2
-\big( C + C_1 \|g\|_{L^{3/(3+\gamma+2s')}_{-\gamma} }\big) \| \la v \ra^{\gamma/2} f\|_{H^{s'}}^2\, .
\end{align}
\end{prop}
\begin{rema}\label{amelio}
It should be noted  that the above coercive estimate is more precise than Theorem 1.2 of \cite{H-C} and more adaptable
to prove the regularity of weak solutions. In fact, the coercive estimate \eqref{coer-uni} is uniform with respect to $g$.
 If $\gamma +4s >0$ and $D(g,g) < \infty$  then $g$ belongs to $L^{3/(3+\gamma+2s')}_{-\gamma}$, provided
that $g \in L^1_\ell$ for a sufficiently large $\ell$.
In fact,  it follows from the proof of Corollary \ref{entropy-dissipation-coer} below that
$D(g,g) < \infty$ implies $ \sqrt g \in H^s_{\gamma/2}$ and hence $\la v\ra^\gamma g \in L^{3/(3-2s)}$
by means of  the Sobolev embedding theorem, which together with
 Lemma \ref{interp-2-lem}   below lead us to this conclusion.
\end{rema}
\begin{proof}
Put
$$
\cC_\gamma(g,\, f) = \iiint_{\RR^3 \times \RR^3 \times \SS^2} b(.) | v-v_*|^\gamma g_* (f'-f)^2dv dv_*d\sigma,
$$
and
note that
$$
\big(Q(g,\, f),\,f\big) =-\frac{1}{2} \cC_\gamma(g,\, f)
+ \frac{1}{2}\iiint \Phi\,b \,\, g_* ( f'^2 -f^2) dv dv_*d\sigma \,.
$$
It follows from the Cancellation Lemma and Remark 6 in \cite{advw}
that
\begin{align*}
\left|\iiint \, b |v-v_*|^\gamma \,\, g_* (f^2 -f'^2) dv dv_*d\sigma\right|
&\lesssim
\left|\iint |v-v_*|^\gamma \,\, g_* f^2 \right|  dv dv_* \\
&\lesssim  \|g\|_{L^1_{|\gamma|}}\|f\|^2_{H^{(-\gamma/2)^+}_{\gamma/2}}\,\,,
\end{align*}
where the last inequality  in the case $\gamma \ge 0$ is trivial. While $\gamma <0$,
this follows from the fact that
\[
|v-v_*|^\gamma \lesssim \la v \ra^\gamma\{ {\bf 1}_{|v-v_*| \ge \la v \ra/2}
+ {\bf 1}_{|v-v_*| < \la v \ra/2} \la v_*\ra^{-\gamma} |v-v_*|^\gamma\},
\]
and the Hardy inequality $\sup_{v_*} \int |v-v_*|^\gamma |F(v)|^2 dv \lesssim \|F\|^2_{H^{-\gamma/2}}$
for $F = \la v \ra^{\gamma/2}f$. Furthermore, it follows from the Hardy-Littlewood-Sobolev inequality that
\begin{align*}
&\left|\iint |v-v_*|^\gamma \,\, g_* f^2 \right|  dv dv_* \lesssim \|g\|_{L^1} \|F\|_{L^2}^2+
\iint \frac{\la v_*\ra^{|\gamma|} g(v_*) F(v)^2}{|v-v_*|^{-\gamma}}dv dv_* \\
&\lesssim  \|g\|_{L^1}\|F\|^2_{L^2} + \|\la v \ra^{|\gamma|} g\|_{L^{3/(3+\gamma+2s') }}
\|F^2\|_{L^{3/(3-2s')}} \lesssim \|g\|_{L^{3/(3+\gamma+2s') }_{|\gamma|}}\|f\|^2_{H^{s'}_{\gamma/2}}       \,\,,
\end{align*}
where we have used the Sobolev embedding in the last inequality.

For the proof of the proposition, it now suffices to consider only the quantity $\cC_\gamma(g,\, f)$. The case $\gamma=0$ is obvious. In fact,
by Corollary 3 and Proposition 2 in \cite{advw},  there exists a $c_0= c_0(D_0, E_0)>0$ depending only on $D_0, E_0 >0$ such that
\begin{equation}\label{coercive-maxwellian}
\cC_0(g,\, f) \ge c_0 \int_{\{|\xi|\ge 1\}} \left||\xi|^s \hat f(\xi)\right|^2 d \xi \, ,\enskip  \mbox{$\forall f \in \mathcal S(\RR^3)$}\,,
\end{equation}
where $\hat f(\xi)$ is the Fourier transform of $f$ with respect to the variable $v\in\RR^3$.
{}From the proof in
\cite{advw}, it should be noted  that \eqref{coercive-maxwellian} holds for any $f \in L^2$ such that  the left hand side
is finite.

We consider the case $\gamma \ne 0$, following the argument used in the proof of Lemma 2
of \cite{advw}.
Choose $R, r_0$ such that \eqref{uni-g} holds.
Let $\varphi_R$ be a non-negative  smooth function not greater  than one, which is 1
for $| v| \ge 4R$ and $0$ for $| v| \le {2R}$.
In view of
\[
\frac{\la v \ra}{4} 
\le|v-v_*|\le 2 \la v \ra \enskip
\mbox{on supp $(\chi_{B(R)})_* \varphi_R$}\,,
\]
we have
\begin{align*}
&4^{|\gamma|}\Phi(|v-v_*|) g_*(f'-f)^2
\ge \big( g \chi_{B(R)}\big)_* \big(\la v \ra^{\gamma/2} \varphi_R\big)^2(f'-f)^2 \\
&\ge \big( g \chi_{B(R)}\big)_* \Big[\frac{1}{2} \Big(\big(\la v \ra^{\gamma/2} \varphi_R f \big)'-
\la v \ra^{\gamma/2} \varphi_R f \Big)^2
 -\Big( \big(\la v \ra^{\gamma/2} \varphi_R\big)'
 - \la v \ra^{\gamma/2} \varphi_R  \Big)^2 {f' }^2 \Big]\,.
\end{align*}
It follows from the mean value theorem that for a $\tau \in (0,1)$
\begin{align*}
\left|\big(\la v \ra^{\gamma/2} \varphi_R\big)'
 - \la v \ra^{\gamma/2} \varphi_R \right| &\lesssim \la v+ \tau(v'-v)\ra^{\gamma/2-1}|v-v_*|\sin\frac{\theta}{2}\\
 &\lesssim \la v_* \ra^{ |\gamma/2-1|} \la v'-v_* \ra^{\gamma/2} \sin\frac{\theta}{2}\\
 &\lesssim \la v_* \ra^{|\gamma/2|+  |\gamma/2-1|} \la v' \ra^{\gamma/2} \sin\frac{\theta}{2},
\end{align*}
because $|v-v_*|/\sqrt 2 \le |v'-v_*| \le |v+ \tau(v'-v) -v_*| \le |v-v_*|$ for $\theta \in [0,\pi/2]$. Therefore,
we have
\begin{align}\label{coer-out}
\cC_\gamma(g,\, f) \ge 2^{-1-2|\gamma|} \cC_0 ( g \chi_{B(R)},\, \varphi_R \la v \ra^{\gamma/2}   f     )
- C_R \|g\|_{L^1} \|f\|^2_{L^2_{\gamma/2}},
\end{align}
for a positive constant $C_R \sim R^{|\gamma|+  |\gamma-2|} $.
For a set $B(4R)$ we take a finite covering
\[
B(4R) \subset \mathop{\cup}_{v_j \in B(4R)} A_j \,, \enskip A_j = \{ v \in \RR^3\,;\, |v-v_j| \le \frac{r_0}{4}\}\,.
\]
For each $A_j$ we choose a non-negative smooth function $ \varphi_{A_j}$ which is  $1$ on $A_j$
and $0$ on $\{|v-v_j|\ge r_0/2\}$.
Note that
\[
\frac{r_0}{2} 
\le|v-v_*|\le 6R  \enskip
\mbox{on supp $(\chi_{B_j(R,r_0)})_* \varphi_{A_j}$}\,.
\]
Then we  have
\begin{align*}
&\Phi(|v-v_*|) g_*(f'-f)^2
\gtrsim \min\{r_0^{\gamma^+}, R^{-(-\gamma)^+}\} \big( g \chi_{B_j(R,r_0)}\big)_* \varphi_{A_j}^2(f'-f)^2 \\
&\gtrsim
R^{-\gamma^+} \min\{r_0^{\gamma^+}, R^{-(-\gamma)^+}\} \big( g \chi_{B_j(R,r_0)}\big)_* \\
&\quad \times
\left [\frac{1}{2} \Big(\big(\la v \ra^{\gamma/2} \varphi_{A_j} f \big)'-
\la v \ra^{\gamma/2} \varphi_{A_j} f \Big)^2
 -\Big( \big(\la v \ra^{\gamma/2} \varphi_{A_j}\big)'
 - \la v \ra^{\gamma/2} \varphi_{A_j}  \Big)^2 {f' }^2 \right]\,.
\end{align*}
Since $\left|\big(\la v \ra^{\gamma/2} \varphi_{A_j}\big)'
 - \la v \ra^{\gamma/2} \varphi_{A_j} \right| \lesssim R^{|\gamma|+1}\la v'\ra^\gamma \sin\theta/2$ if $|v_*|\le R$, we obtain
\begin{align}\label{coer-inner}
\cC_\gamma(g,\, f)
&\gtrsim
 \min\{(r_0/R)^{\gamma^+}, R^{-(-\gamma)^+}\}
\cC_0 ( g \chi_{B_j(R,r_0)},\, \varphi_{A_j} \la v \ra^{\gamma/2}   f     )\\
& \qquad \qquad \qquad - C'_{R,r_0} \|g\|_{L^1} \|f\|^2_{L^2_{\gamma/2}}, \notag
\end{align}
for a positive constant $C'_{R,r_0} \sim R^{2+2|\gamma|}$.
It follows from \eqref{coercive-maxwellian}, \eqref{coer-out} and \eqref{coer-inner}  that
there exist $c'_0, C, C' >0$ depending only on $D_0, E_0$ such that
\begin{align}\label{coer-uni-positive}
\cC_\gamma(g,\, f) &\ge c'_0\Big(
\|\la D \ra^s \varphi_R \la v \ra^{\gamma/2} f\|^2 + \sum_j \|\la D \ra^s \varphi_{A_j} \la v \ra^{\gamma/2} f\|^2\Big)
-C \|f\|_{L^2_{\gamma/2}}^2\\
&\ge c'_0 \|\la v \ra^{\gamma/2} f\|^2_{H^s} - C' \|f\|_{L^2_{\gamma/2}}^2, \notag
\end{align}
because $\varphi_R^2 + \sum_j \varphi_{A_j}^2 \ge 1$ and commutators
$[\la D\ra^s, \varphi_R]$, $[\la D \ra^s, \varphi_{A_j}]$ are $L^2$ bounded operators.
\end{proof}

\begin{rema}\label{f-L2}
\eqref{coer-uni-positive} holds for any $f \in L^2_{\gamma/2}$  such that ${\cC}_\gamma(g,f)$ is finite, because of  the remark after \eqref{coercive-maxwellian}.
Similarly, \eqref{coer-uni} holds for  any $f \in L^2_{\gamma/2}$
if $\gamma \ge 0$ and if its left hand side is finite.
\end{rema}

\begin{coro}\label{entropy-dissipation-coer}
Let $f(t) \in L^1_{\max\{2, \gamma\}} \cap L\log L$ be a weak solution.   Suppose that the cross section $B$ is the same as in
Propostion \ref{coercive-uniform}. Assume that for a $T>0$ we have
\begin{align}\label{entropy-dissipation-estimate}
 \int_0^T D(f(\tau),f(\tau) ) d \tau < \infty\,.
\end{align}
Then there exist positive constants $c_f$ and $C_f >0$ such that
\begin{equation}\label{entopy}
c_f \int_0^T \|\sqrt{f(\tau)}\|^2_{H^s_{\gamma/2}}d\tau \leq \int_0^T D(f(\tau), f(\tau))d\tau
+C_f \int_0^T\|f(\tau) \|_{L^1_{\gamma^+}} d\tau \,.
\end{equation}
\end{coro}
\begin{proof}
We first consider the case $\gamma<0$. Note
\begin{align*}
D(f,f) &= -\iiint B \big (f'_* f' -f_* f \big )\log f dv dv_* d\sigma \\
&=\frac 1 4 \iiint B \, \big(  f'f'_* - f f_*\big) \big (\log f' f'_* -\log f f_* \big)  dv dv_* d\sigma \\
& \ge \frac 1 4  \iiint b(\cdot)  \la v-v_*\ra^\gamma \, \big(  f'f'_* - f f_*\big) \big (\log f' f'_* -\log f f_* \big)  dv dv_* d\sigma,
\end{align*}
because $(x-y)(\log x - \log y) \ge 0$ and $\Phi(|v-v_*|) \ge \la v-v_*\ra^\gamma$. Then we have
\begin{align*}
D(f,f) &\ge  -\iiint b(\cdot) \la v-v_*\ra^\gamma  \big (f'_* f' -f_* f \big )\log f dv dv_* d\sigma \\
& = \iiint b(\cdot) \la v-v_*\ra^\gamma  f_* \left(
f \log \frac{f}{f'} -f + f'\right) dv dv_* d\sigma \\
&\qquad + \iiint b(\cdot) \la v-v_*\ra^\gamma  f_* \left(
f -  f'\right) dv dv_* d\sigma \\
& \ge \iiint b(\cdot) \la v-v_*\ra^\gamma  f_* \left(
\sqrt {f'} - \sqrt f \right)^2  dv dv_* d\sigma - C \|f\|_{L^1}^2 \,,
\end{align*}
where we have used $x \log(x/y) - x + y \ge (\sqrt x - \sqrt y)^2$ and the Cancellation Lemma in the last inequality,
as the same as in the proof of Theorem 1 in \cite{advw}. Since the proof of Propostion \ref{coercive-uniform}
still works  with  $\Phi$ replaced by $\la v-v_*\ra^\gamma$, we obtain the desired estimate
in view of Remark \ref{f-L2}. The case $\gamma \ge 0$ is easier because we do not need to replace
$\Phi$ by $\la v-v_*\ra^\gamma$ when Cancellation Lemma is applied.
\end{proof}

\section{Mollifier and commutator estimate}\label{section2}
\setcounter{equation}{0}

Since the weak solution is only in $L^1$, we can not use it directly as
a test function in the definition of weak solution to get the energy
estimate. To overcome this, we need to mollify
it by some suitable mollifiers so that to consider the commutators between the
mollifiers and the collision operator becomes necessary.

Let $\lambda , N_0 \in \RR$, $\delta >0$ and put
\begin{align}\label{mollifier}
M_\lambda^\delta(\xi) = \frac{\la \xi \ra^{\lambda}}{(1+ \delta \la \xi \ra )^{N_0}}\,,\, \enskip \la \xi \ra = (1+|\xi|^2)^{1/2}\,.
\end{align}
Then $M_\lambda^\delta(\xi)$ belongs to the symbol  class $S^{\lambda -N_0}_{1,0}$ of pseudo-differential operators and
belongs to $S^{\lambda}_{1,0} $ uniformly with respect to $\delta \in ]0,1]$.
 $M_\lambda^\delta(D_v)$ denotes the associated pseudo-differential operator.
By direct calculation we see that for any $\alpha$
there exists
a $C_\alpha >0$ independent of $\delta$   such that
\begin{equation}\label{der-es}
\Big|\partial^\alpha_\xi  M_\lambda^\delta(\xi)\Big| \le C_\alpha M_\lambda^\delta(\xi) \la \xi \ra^{-|\alpha|}\,.
\end{equation}

\begin{lemm}\label{basic-inequality}
There exists a constant $C >0$ independent of $\delta$ such that
\begin{align}\label{M-inequality2}
&\left|M_\lambda^\delta (\xi) - M_\lambda^\delta (\xi-\xi_*)\right|\\
&\leq C \la \xi\ra^\lambda {\bf 1}_{\la \xi_* \ra \ge \sqrt 2 |\xi|} +C M_\lambda^\delta (\xi-\xi_*)\left\{{\bf 1}_{\la \xi_* \ra \ge |\xi|/2}+  \frac{\la \xi_*\ra}{\la \xi \ra }{\bf 1}_{|\xi|/2>\la \xi_* \ra }\right\}\notag\\
& +C M_\lambda^\delta (\xi-\xi_*)\left( \frac{M_\lambda^\delta (\xi_*)\big(1 + \delta \la \xi - \xi_*\ra\big)^{N_0}}{\la \xi -\xi_*\ra^\lambda}
 \right) {\bf 1}_{   \sqrt 2 |\xi| >\la \xi_* \ra \ge   |\xi|/2  }
\,.  \notag
\end{align}
And if $p \geq N_0 - \lambda$
\begin{align}\label{M-inequality}
&\left|M_\lambda^\delta (\xi) - M_\lambda^\delta (\xi-\xi_*)\right|\leq C M_\lambda^\delta (\xi-\xi_*)\left\{ \Big(
\frac{\la \xi_*\ra}{\la \xi \ra }\Big)^p {\bf 1}_{\la \xi_* \ra \ge \sqrt 2 |\xi|} \right.\\
&\quad + \left( \frac{M_\lambda^\delta (\xi_*)\big(1 + \delta \la \xi - \xi_*\ra\big)^{N_0}}{\la \xi -\xi_*\ra^\lambda}
+1 \right) {\bf 1}_{   \sqrt 2 |\xi| >\la \xi_* \ra \ge   |\xi|/2  }
+\left.  \frac{\la \xi_*\ra}{\la \xi \ra }{\bf 1}_{|\xi|/2>\la \xi_* \ra }\right\}\,.  \notag
\end{align}
\end{lemm}

\begin{proof}
We first  note
\begin{equation}\label{equivalence-relation-spatiall-homo}
\left \{ \begin{array}{ll}
\la \xi \ra \lesssim \la \xi_* \ra \sim \la \xi-\xi_*\ra,   &\mbox{on supp ${\bf 1}_{\la \xi_* \ra\geq \sqrt 2 |\xi|}$},\\
\la \xi \ra \sim \la \xi-\xi_*\ra,   &\mbox{on supp ${\bf 1}_{\la \xi_*\ra \leq |\xi |/2 } $}, \\
\la \xi \ra \sim \la \xi_* \ra \gtrsim   \la \xi-\xi_*\ra,  &
\mbox{on supp ${\bf 1}_{\sqrt 2 |\xi| \geq \la \xi_*\ra \geq | \xi|/2 }$\,.}
\end{array}
\right.
\end{equation}
Since $\la \xi \ra^p M_\lambda^\delta (\xi)$ is increasing function of $\la \xi\ra$, we have
\begin{equation*}
\la \xi \ra^p M_\lambda^\delta (\xi) \lesssim \la \xi_* \ra^p M_\lambda^\delta (\xi_*)
\sim \la \xi_* \ra^p M_\lambda^\delta (\xi- \xi_*) \mbox{\enskip on supp ${\bf 1}_{\la \xi_* \ra\geq \sqrt 2 |\xi|}$}\, ,
\end{equation*}
and trivially,
$$
M_\lambda^\delta (\xi) \leq \la \xi \ra^\lambda.
$$
Note  that
\begin{equation*}
M_\lambda^\delta (\xi) \sim M_\lambda^\delta (\xi_*) \sim
M_\lambda^\delta (\xi-\xi_*) \frac{M_\lambda^\delta (\xi_*)\big(1 + \delta \la \xi - \xi_*\ra\big)^{N_0}}{\la \xi -\xi_*\ra^\lambda}
\mbox{\enskip on supp ${\bf 1}_{\sqrt{2}|\xi|\geq\la \xi_*\ra \geq |\xi |/2 } $}.
\end{equation*}
By the mean value theorem,  we have
\begin{align*}
\left|M_\lambda^\delta (\xi) - M_\lambda^\delta (\xi-\xi_*)\right| &\le
\int_0^1 |\big(\nabla_\xi  M_\lambda^\delta\big) (\xi + \tau(\xi-\xi_*))|d\tau |\xi_*| \\
&  \lesssim  M_\lambda^\delta (\xi-\xi_*)  \frac{\la \xi_*\ra}{\la \xi \ra }\quad
\mbox{on supp ${\bf 1}_{\la \xi_*\ra \leq |\xi |/2 } $}  \notag .
\end{align*}
Here we have used \eqref{der-es} and the second formula of \eqref{equivalence-relation-spatiall-homo}.
The above estimates imply \eqref{M-inequality} and \eqref{M-inequality2}.
\end{proof}

For the  kinetic factor
$|v-v_*|^\gamma$, we need to take into account the singular behavior
close to $|v-v_*|=0$ except $\gamma=0$. Therefore, we decompose the kinetic factor in two
parts. Let $0\leq \phi (z)\leq 1$ be a smooth radial function with
value $1$ for ~$z$ close to $0$, and $0$ for large values of ~$z$. Set
$$
\Phi_\gamma (z) = \Phi_\gamma (z) \phi (z) + \Phi_\gamma (z) (1-\phi (z)) = \Phi_c (z) + \Phi_{\bar c} (z).
$$
And then correspondingly we can write
$$
Q (f, g) = Q_c (f, g) + Q_{\bar c} (f, g),
$$
where the kinetic factor in the collision operator is defined
according to the decomposition respectively. Since $\Phi_{\bar c}
(z)$ is smooth, and $\Phi_{\bar c} (z)\lesssim \tilde{\Phi}_\gamma(z)$, where $\tilde{\Phi}_\gamma(|\,z\,|)=(1+|z|^2)^{\gamma/2}$ is the regular kinetic factor studied in \cite{amuxy3}. Then
$Q_{\bar c} (f, g)$ has similar properties as for
$Q_{\tilde{\Phi}_\gamma} (f, g)$ as regard to the upper bound and commutator estimations.
We recall the Proposition 2.9 of \cite{amuxy3}.

\begin{prop}\label{prop2.9_amuxy3}
Let $\lambda \in \RR$ and  $M(\xi)$ be a positive symbol of
pseudo-differential operator in $S^{\lambda}_ {1,0}$ in the form of
$M(\xi) = \tilde M(|\xi|^2)$.  Assume that, there exist
constants $c, C>0$ such that
 for any $s, \tau>0$
$$
c^{-1}\leq \frac{s}{\tau}\leq c \,\,\,\,\,\,\mbox{implies}
\,\,\,\,\,\,\,\,C^{-1}\leq \frac{\tilde M(s)}{ \tilde M(\tau)}\leq C,
$$
and  $M(\xi)$ satisfies
$$
|M^{(\alpha)}(\xi)| = |\partial_\xi^\alpha M(\xi)| \leq C_{\alpha}
M(\xi) \la \xi \ra^{-|\alpha|}\, ,
$$
for any $\alpha\in\NN^3$.  Then,  if $0<s<1/2$, for any $N >0$ there exists a $C_N >0$ such that
\begin{eqnarray}\label{10.8-2}
\left|( M(D_v) Q_{\bar c}(f,\, g)-  Q_{\bar c}(f,\, M(D_v) g) ,\,\, h)_{L^2}\right | \hskip4cm\\
\hskip3cm \leq C_N \|f\|_{L^1_{\gamma^+}} \Big(\|M(D_v)\,
g\|_{L^2_{\gamma^+}} + \| g \|_{H^{\lambda-N}_{\gamma^+}} \Big)
\|h\|_{L^2}.\nonumber
\end{eqnarray}
Furthermore, if $1/2 < s <1$,  for any $N>0$ and any $\varepsilon >0$ , there exists
a  $C_{N, \varepsilon}>0 $ such that
\begin{align}\label{10.8-3}
\left |(M(D_v) Q_{\bar c}(f,\, g)-Q_{\bar c}(f,\, M(D_v) g),\,\, h)_{L^2} \right |  \hskip4cm \\
\leq C_{N, \varepsilon} \|f\|_{L^1_{(2s+ \gamma-1)^+}} \Big(
\|M(D_v) g\|_{H^{2s-1+\varepsilon} _{(2s+ \gamma-1)^+}}  +
\|g\|_{H^{\lambda-N}_{\gamma^+}}\Big)
 \|h\|_{L^2}\, . \nonumber
\end{align}
When $s = 1/2$ we have the same  estimate  as (\ref{10.8-3}) with
$(2s+ \gamma-1)$ replaced by $(\gamma+ \kappa)$ for any small
$\kappa >0$.
\end{prop}

\begin{rema}\label{hard-poten}
In the case $\gamma >0$ and $0<s<1/2$, it follows from Lemma 3.1 of \cite{HMUY} and its proof that \eqref{10.8-2} can be replaced by
\begin{eqnarray*}
\left|( M(D_v) Q_{\bar c}(f,\, g)-  Q_{\bar c}(f,\, M(D_v) g) ,\,\, h)_{L^2}\right | \hskip4cm\\
\hskip3cm \leq C_N \|f\|_{L^1_{\gamma}} \Big(\|M(D_v)\,
g\|_{L^2_{\gamma/2}} + \| g \|_{H^{\lambda-N}_{\gamma/2}} \Big)
\|h\|_{L^2_{\gamma/2}}.\nonumber
\end{eqnarray*}
\end{rema}
{}From now on, we concentrate on the study for the singular part $Q_{c}(f, g)$.

\begin{prop} \label{commu-molli}
Assume that  $0< s<1, \gamma+2s>0$.
Let  $0<s'<s$ satisfy $
\gamma+2s'>0$ and $2s'\geq (2s -1)^+\, .$
If
\begin{equation}\label{restriction}
5+ \gamma \geq 2(N_0-\lambda)\, ,
\end{equation}
then we have

\noindent
1) If $s' + \lambda <3/2$,   then
\begin{align*}
\Big
| \Big (M_\lambda^\delta (D_v) \, Q_c (f,  g) - Q_c (f,  M_\lambda^\delta (D_v)\, g)  , h \Big ) \Big |
\lesssim  \| f\|_{L^1} || M_\lambda^\delta(D_v)g||_{H^{s'}} \, || h||_{H^{s'}}\,.
\end{align*}
2) If $s'+\lambda \ge 3/2$,
then
\begin{align*}
\Big
| \Big (M_\lambda^\delta(D_v) \, Q_c (f,  g) - &Q_c (f,  M_\lambda^\delta(D_v)\, g)  , h \Big ) \Big |\\
&\lesssim  \Big(  \| f\|_{L^1} +  \| f\|_{{ H^{ (\lambda +s'-3)^+}}} \Big)|| M_\lambda^\delta(D_v)g||_{H^{s'}} \, || h||_{H^{s'}}\,.
\end{align*}
Furthermore, if  $s>1/2$ and $\gamma >  -1$, then the assumption  \eqref{restriction} can be relaxed to
\begin{equation}\label{slight}
4+ \gamma + 2s > 2(N_0-\lambda)\, .
\end{equation}
\end{prop}

\begin{proof}
For the proof  we shall follow some of arguments from \cite{amuxy4-1}. By using the formula from
the Appendix of \cite{advw}, we have
\begin{align*}
( Q_c(f, g), h ) =& \iiint_{ \RR^3\times\RR^3\times\SS^2} b \Big({\xi\over{ | \xi |}} \cdot \sigma \Big) [ \hat\Phi_c (\xi_* - \xi^- ) - \hat \Phi_c (\xi_* ) ] \\
& \qquad \qquad  \times \hat f (\xi_* ) \hat g(\xi - \xi_* ) \overline{{\hat h} (\xi )} d\xi d\xi_*d\sigma \,,
\end{align*}
where $\xi^-=\frac 1 2 (\xi-|\xi|\sigma)$.
Therefore
\begin{align*}
 \Big (M_\lambda^\delta(D) \, Q_c (f,  g)& -  Q_c (f,  M_\lambda^\delta(D)\, g)  , h \Big ) \\
= &\iiint b \Big({\xi\over{ | \xi |}} \cdot \sigma \Big) [ \hat\Phi_c (\xi_* - \xi^- ) - \hat \Phi_c (\xi_* ) ] \\
&\quad \times
\Big(M_\lambda^\delta (\xi) - M_\lambda^\delta (\xi-\xi_*)\Big)
\hat f (\xi_* ) \hat g(\xi - \xi_* ) \overline{{\hat h} (\xi )} d\xi d\xi_*d\sigma \\
= & \iiint_{ | \xi^- | \leq {1\over 2} \la \xi_*\ra }  \cdots\,\, d\xi d\xi_*d\sigma
+ \iiint_{ | \xi^- | \geq {1\over 2} \la \xi_*\ra } \cdots\,\, d\xi d\xi_*d\sigma \,\\
=& \cA_1(f,g,h)  +  \cA_2(f,g,h) \,\,.
\end{align*}
Then, we write $\cA_2(f,g,h)$ as
\begin{align*}
\cA_2 &=  \iiint b \Big({\xi\over{ | \xi |}} \cdot \sigma \Big) {\bf 1}_{ | \xi^- | \ge {1\over 2}\la \xi_*\ra }
\hat\Phi_c (\xi_* - \xi^- ) \cdots 
d\xi d\xi_*d\sigma \\
&\quad - \iiint b \Big({\xi\over{ | \xi |}} \cdot
 \sigma \Big){\bf 1}_{ | \xi^- | \ge {1\over 2}\la \xi_*\ra } \hat \Phi_c (\xi_* )
\cdots 
 d\xi d\xi_*d\sigma \\
&= \cA_{2,1}(f,g,h) - \cA_{2,2}(f,g,h)\,.
\end{align*}
On the other hand,  for $\cA_1$ we use the Taylor expansion of $\hat \Phi_c$ of
order $2$ to have
$$
\cA_1 = \cA_{1,1} (f,g,h) +\cA_{1,2} (f,g,h),
$$
where
\begin{align*}
\cA_{1,1} &= \iiint b\,\, \xi^-\cdot (\nabla\hat\Phi_c)( \xi_*)
{\bf 1}_{ | \xi^- | \leq {1\over 2} \la \xi_*\ra }
\Big(M_\lambda^\delta (\xi) - M_\lambda^\delta (\xi-\xi_*)\Big)\\
& \qquad \times
\hat f (\xi_* ) \hat g(\xi - \xi_* ) \bar{\hat h} (\xi ) d\xi d\xi_*d\sigma,
\end{align*}
and $\cA_{1,2} (f,g,h)$ is the remaining term corresponding to the second order term in the Taylor expansion of $\hat\Phi_c$.

We first consider $\cA_{1,1}$.
By writing
\[
\xi^- = \frac{|\xi|}{2}\left(\Big(\frac{\xi}{|\xi|}\cdot \sigma\Big)\frac{\xi }{|\xi|}-\sigma\right)
+ \left(1- \Big(\frac{\xi}{|\xi|}\cdot \sigma\Big)\right)\frac{\xi}{2},
\]
we see that the integral corresponding to the first term on the right hand side vanishes because of the symmetry
on $\SS^2$.
Hence, we have
\[
\cA_{1,1}= \iint_{\RR^6} K(\xi, \xi_*) \Big(M_\lambda^\delta (\xi) - M_\lambda^\delta (\xi-\xi_*)\Big)
\hat f (\xi_* ) \hat g(\xi - \xi_* ) \bar{\hat h} (\xi ) d\xi d\xi_* \,,
\]
where
\[
K(\xi,\xi_*) = \int_{\SS^2}
 b \Big({\xi\over{ | \xi |}} \cdot \sigma \Big)
\left(1- \Big(\frac{\xi}{|\xi|}\cdot \sigma\Big)\right)\frac{\xi}{2}\cdot
(\nabla\hat\Phi_c)( \xi_*)
{\bf 1}_{ | \xi^- | \leq {1\over 2} \la \xi_*\ra } d \sigma \,.
\]
Note that $| \nabla \hat \Phi_c (\xi_*) | \lesssim \frac{1}{\la
\xi_*\ra^{3+\gamma +1}}$, from the Appendix of \cite{amuxy4-1}. If $\sqrt 2
|\xi| \leq \la \xi_* \ra$, then $\sin(\theta/2) $ $| \xi|= |\xi^-| \leq \la \xi_* \ra/2$
because $0 \leq \theta \leq \pi/2$, and we have
\begin{align*}
|K(\xi,\xi_*)| &\lesssim \int_0^{\pi/2} \theta^{1-2s} d \theta\frac{ \la \xi\ra}{\la \xi_*\ra^{3+\gamma +1}}
\lesssim \frac{1  }{\la \xi_*\ra^{3+\gamma}}\left(
\frac{\la \xi \ra}{\la \xi_*\ra}\right) \,.
\end{align*}
On the other hand, if $\sqrt 2 |\xi| \geq \la \xi_* \ra$, then
\begin{align*}|K(\xi,\xi_*)| &\lesssim \int_0^{\pi\la \xi_*\ra /(2|\xi|)} \theta^{1-2s} d \theta\frac{ \la \xi\ra}{\la \xi_*\ra^{3+\gamma +1}}
\lesssim \frac{1  }{\la \xi_*\ra^{3+\gamma}}\left(
\frac{\la \xi \ra}{\la \xi_*\ra}\right)^{2s-1}\,.
\end{align*}
Hence we obtain
\begin{align}\label{later-use1}
|K(\xi,\xi_*)| &\lesssim \frac{1  }{\la \xi_*\ra^{3+\gamma}}\left\{
\left( \frac{\la \xi \ra}{\la \xi_*\ra}\right){\bf 1}_{ \la \xi_*
\ra \geq
\sqrt 2 |\xi| } \right.\\
&\qquad\left.+{\bf 1}_{ \sqrt 2 |\xi| \geq  \la \xi_* \ra \geq
|\xi|/2} + \left( \frac{\la \xi \ra}{\la \xi_*\ra}\right)^{2s-1}
{\bf 1}_{ |\xi|/2 \ge \la \xi_* \ra }\right\}\,. \notag
\end{align}
Similar to $\cA_{1,1}$,
we  can also write
\[
\cA_{1,2}= \iint_{\RR^6} \tilde K(\xi, \xi_*) \Big(M_\lambda^\delta (\xi) - M_\lambda^\delta (\xi-\xi_*)\Big)
\hat f (\xi_* ) \hat g(\xi - \xi_* ) \bar{\hat h} (\xi ) d\xi d\xi_* \,,
\]
where
\[
\tilde K(\xi,\xi_*) = \int_{\SS^2}
 b \Big({\xi\over{ | \xi |}} \cdot \sigma \Big)
\int^1_0(1-\tau)  (\nabla^2\hat \Phi_c) (\xi_* -\tau\xi^- ) \cdot\xi^- \cdot\xi^-
{\bf 1}_{ | \xi^- | \leq {1\over 2} \la \xi_*\ra } d\tau  d \sigma\,.
\]
Again from the Appendix of \cite{amuxy4-1}, we have
$$
| (\nabla^2\hat \Phi_c) (\xi_* -\tau\xi^- ) | \lesssim {1\over{\la  \xi_* -\tau \xi^-\ra^{3+\gamma +2}}}
\lesssim
 {1\over{\la \xi_*\ra^{3+\gamma +2}}},
$$
because $|\xi^-| \leq \la \xi_*\ra/2$, which leads to
\begin{align}\label{later-use2}
|\tilde K(\xi,\xi_*)| &\lesssim \frac{1  }{\la
\xi_*\ra^{3+\gamma}}\left\{ \left( \frac{\la \xi \ra}{\la
\xi_*\ra}\right)^2 {\bf 1}_{ \la \xi_* \ra \geq \sqrt 2 |\xi|
}\right.\\
&\qquad\left. +{\bf 1}_{ \sqrt 2 |\xi| \geq  \la \xi_* \ra \geq
|\xi|/2} + \left( \frac{\la \xi \ra}{\la \xi_*\ra}\right)^{2s} {\bf
1}_{ |\xi|/2 \ge \la \xi_* \ra }\right\}\,. \notag
\end{align}
It follows from  \eqref{M-inequality} of Lemma \ref{basic-inequality}, \eqref{later-use1} and \eqref{later-use2} that
if $p = N_0- \lambda$, then
$$
|\cA_{1}|\lesssim  |\cA_{1,1}| + |\cA_{1,2}|
\lesssim A_1+ A_2 + A_3,
$$
where
\begin{equation}\label{A_1}
A_1=\iint_{\RR^6}
\left |\frac{\hat f(\xi_*)}{\la \xi_*\ra^{3+\gamma}}\right | \left |M_\lambda^\delta (\xi-\xi_*)\hat g(\xi -\xi_*)\right | |\hat h(\xi)|\left(
\frac{\la \xi_* \ra}{\la \xi\ra}\right)^{p-1} {\bf 1}_{ \la \xi_* \ra \geq \sqrt 2 |\xi| }d\xi_* d\xi\, ,
\end{equation}
and
\begin{align*}
A_2=&
\iint_{\RR^6}
\left |\frac{\hat f(\xi_*)}{\la \xi_*\ra^{3+\gamma}}\right | \left |M_\lambda^\delta (\xi-\xi_*)\hat g(\xi -\xi_*)\right | |\hat h(\xi)|\\
&\times\left( \frac{M_\lambda^\delta (\xi_*)\Big(1 + \big(\delta \la \xi - \xi_*\ra\big)^{N_0}\Big)}{\la \xi -\xi_*\ra^\lambda}
+1 \right) {\bf 1}_{   \sqrt 2 |\xi| >\la \xi_* \ra \ge   |\xi|/2  }d\xi_*d\xi\, ;\\
A_3=&\iint_{\RR^6}
\left |\frac{\hat f(\xi_*)}{\la \xi_*\ra^{3+\gamma}}\right | \left |M_\lambda^\delta (\xi-\xi_*)\hat g(\xi -\xi_*)\right | |\hat h(\xi)|
\left(  \frac{\la \xi\ra}{\la \xi_* \ra } \right)^{2s-1}  {\bf 1}_{|\xi|/2>\la \xi_* \ra } d\xi_* d\xi\, .
\end{align*}
Putting $\hat G(\xi) = \la \xi \ra^{s'} M_\lambda^\delta (\xi) \hat g(\xi)$ and $\hat H(\xi) = \la \xi \ra^{s'}\hat h(\xi)$,  then
we have
\begin{align*}
&\left| A_1\right|^2 \lesssim \|\hat f \|^2_{L^\infty}
\left(\int_{\RR^3} \frac{d\xi_*}{ \la \xi_*\ra^{3+ \gamma +2s'}}
\int_{\RR^3_\xi} |\hat H(\xi)|^2d\xi\right)\\
& \times \left(\int_{\RR^3}  \frac{d\xi}{ \la \xi \ra^{3+ \gamma +2s'}}
\int_{\RR^3} \left(\frac{\la \xi\ra}{\la \xi_*\ra}\right)^{3+\gamma -2(p-1)}
{\bf 1}_{ \la \xi_* \ra \geq \sqrt 2 |\xi| } |\hat G ( \xi -\xi_*)|^2 d\xi_*
\right)\notag \\
& \qquad
\lesssim
\|f\|_{L^1}^2 \|M_\lambda^\delta g\|_{H^{s'}}^2 \|h\|_{H^{s'}}^2\,, \notag
\end{align*}
because  $\gamma + 2s' >0$, and $3+\gamma -2(p-1) \ge 0$ from \eqref{restriction}.
Here we have used the fact that $\la \xi_* \ra \sim \la \xi-\xi_*\ra$ if $\la \xi_* \ra \geq \sqrt 2 |\xi|$.

We consider the case  $ s >1/2, \gamma >-1$.
For  $s> s' >1/2$ we have
\begin{align*}
&\left| A_1\right|^2 \lesssim \|\hat f \|^2_{L^\infty}
\left(\int_{\RR^3} \frac{d\xi_*}{ \la \xi_*\ra^{3+ \gamma +1}}
\int_{\RR^3_\xi} |\hat H(\xi)|^2d\xi\right)\\
& \times \left(\int_{\RR^3}  \frac{d\xi}{ \la \xi \ra^{3+ \gamma +1}}
\int_{\RR^3} \left(\frac{\la \xi\ra}{\la \xi_*\ra}\right)^{3+\gamma +(2s'-1)-2(p-1)}
\frac{{\bf 1}_{ \la \xi_* \ra \geq \sqrt 2 |\xi| }}{\la \xi \ra^{2(2s'-1)}}
 |\hat G ( \xi -\xi_*)|^2 d\xi_*
\right)\notag \\
& \qquad
\lesssim
\|f\|_{L^1}^2 \|M_\lambda^\delta g\|_{H^{s'}}^2 \|h\|_{H^{s'}}^2\,, \notag
\end{align*}
if $3+ \gamma +(2s'-1)-2(p-1) >0$.
Thus \eqref{restriction} can be relaxed to \eqref{slight} to get the desired estimate for $A_1$. Here we remark that \eqref{restriction} or \eqref{slight} are only required to estimate the part $A_1$.

Noting the third formula of \eqref{equivalence-relation-spatiall-homo}, we get
\begin{align*}
\left| A_2 \right|^2 \lesssim&
\left \{\int_{\RR^3}\frac{|\hat f(\xi_*)|^2 d\xi_*}{
\la \xi_*\ra^{6+2\gamma +2s'}}
\int_{\la \xi -\xi_*\ra \lesssim \la \xi_* \ra} \left(
\frac{\la \xi_* \ra^{2\lambda}}{\la \xi- \xi_* \ra^{2(\lambda+s')}}\right. \right.\\
&\quad \left. \left. +    \frac{\la \xi_* \ra^{2(\lambda-N_0)}}{\la \xi- \xi_* \ra^{2(\lambda-N_0 +s')}}        +
\frac{1}{\la \xi -\xi_* \ra^{2s'}} \right) d\xi \right\} \notag \\
&\quad \times \left( \iint_{\RR^6}|\hat G ( \xi -\xi_*)|^2 |\hat H(\xi)|^2 d\xi d\xi_* \right) \,.\notag
\end{align*}
If $\lambda + s' <3/2$, then
\begin{align*}
\left| A_2 \right|^2& \lesssim
\int_{\RR^3}\frac{|\hat f(\xi_*)|^2} {
\la \xi_*\ra^{3+2(\gamma +2s')}}
d\xi_* \|M_\lambda^\delta g\|_{H^{s'}}^2 \|h\|_{H^{s'}}^2\\
&\lesssim
\|f\|_{L^1}^2 \|M_\lambda^\delta g\|_{H^{s'}}^2 \|h\|_{H^{s'}}^2\,.
\end{align*}
If $\lambda + s' \ge 3/2$, then
\begin{align*}
\left| A_2 \right|^2& \lesssim
\int_{\RR^3}\frac{|\hat f(\xi_*)|^2 \la \xi_* \ra^{2(\lambda+ s' + \varepsilon)} } {
\la \xi_*\ra^{6+2(\gamma +2s')}}
d\xi_* \|M_\lambda^\delta g\|_{H^{s'}}^2
\|h\|_{H^{s'}}^2 \\
&\lesssim
\|f\|_{H^{\lambda +s' -3}}^2 \|M_\lambda^\delta g\|_{H^{s'}}^2 \|h\|_{H^{s'}}^2\,.
\end{align*}

Since  $2s' \ge 2s-1$ and $\gamma + 2s'>0$, we have
\begin{align*}
&\left| A_3 \right|^2 \lesssim \|\hat f \|^2_{L^\infty}
\left(\int_{\RR^3} \frac{d\xi_*}{ \la \xi_*\ra^{3+ \gamma + 2s'}}
\int_{\RR^3} |\hat H(\xi)|^2d\xi\right)\\
& \times \left(\int_{\RR^3}  \frac{d\xi_*}{ \la \xi_* \ra^{3+ \gamma +2s'}}
\int_{\RR^3} \left(\frac{\la \xi_*\ra}{\la \xi \ra}\right)^{2\{2s'-(2s-1)\}}
{\bf 1}_{ |\xi|/2 \ge \la \xi_* \ra } |\hat G ( \xi -\xi_*)|^2 d\xi
\right)\notag \\
& \qquad
\lesssim
\|f\|_{L^1}^2 \|M_\lambda^\delta g\|_{H^{s'}}^2 \|h\|_{H^{s'}}^2\, . \notag
\end{align*}
The above four estimates yield the desired estimate for $\cA_1(f,g,h)$.

Next we consider $\cA_2(f,g,h) = \cA_{2,1}(f,g,h) -
\cA_{2,2}(f,g,h)$. The fact that $|\xi^-|= |\xi| \sin(\theta/2) \geq
\la \xi_*\ra/2$ and $\theta \in [0,\pi/2]$ imply $\sqrt 2 |\xi| \geq
\la \xi_*\ra$. Write
\[
\cA_{2,j}= \iint_{\RR^6} K_j(\xi, \xi_*) \Big(M_\lambda^\delta (\xi) - M_\lambda^\delta (\xi-\xi_*)\Big)
\hat f (\xi_* ) \hat g(\xi - \xi_* ) \bar{\hat h} (\xi ) d\xi d\xi_* \,.
\]
Then we have
\begin{align*}
&|K_2(\xi, \xi_*)| = \left|\int  b \Big({\xi\over{ | \xi |}} \cdot \sigma \Big)\hat \Phi_c(\xi_*) {\bf 1}_{ | \xi^- | \ge {1\over 2}\la \xi_*\ra } d\sigma\right|\\
& \lesssim  {1\over{\la \xi_* \ra^{3+\gamma }}} \frac{\la  \xi\ra^{2s} }{\la \xi_*\ra^{2s}}{\bf 1}_{\sqrt 2 |\xi| \geq \la \xi_* \ra} \notag \\
 &    \lesssim
\frac{1  }{\la \xi_*\ra^{3+\gamma}}\left\{
{\bf 1}_{ \sqrt 2 |\xi| \geq  \la \xi_* \ra \geq |\xi|/2}
+
\left(
\frac{\la \xi \ra}{\la \xi_*\ra}\right)^{2s}
{\bf 1}_{ |\xi|/2 \ge \la \xi_* \ra }\right\}  \notag \,,
\end{align*}
which shows the desired estimate for $\cA_{2,2}$, by  exactly the same way as
the estimation on $A_2$ and $A_3$.

As for $\cA_{2,1}$, it suffices to work under the condition
 $|\xi_* \cdot \xi^-| \ge \frac1 2 |\xi^-|^2$.
In fact, on the complement of this
set, we have
 $|\xi_* -\xi^-| > | \xi_*|$, and $\hat \Phi_c(\xi_*-\xi^-)$ is
the  the same as $\hat \Phi_c(\xi_*)$.
Therefore, we consider $\cA_{2,1,p}$ which is defined by replacing $K_1(\xi, \xi_*)$ by
\[
K_{1,p}(\xi,\xi_*) = \int_{\SS^2}
 b \Big({\xi\over{ | \xi |}} \cdot \sigma \Big)
\hat \Phi_c ( \xi_*-\xi^-)
{\bf 1}_{ | \xi^- | \geq {1\over 2} \la \xi_*\ra }{\bf 1}_{| \xi_* \,\cdot\,\xi^-| \ge {1\over 2} | \xi^-|^2} d \sigma \,.
\]
By noting
\[
{\bf 1}= {\bf 1}_{\la \xi_* \ra \geq |\xi|/2} {\bf 1}_{\la\xi-\xi_* \ra \leq{2}\la \xi_* - \xi^- \ra}
+  {\bf 1}_{\la \xi_* \ra \geq |\xi|/2} {\bf 1}_{\la\xi-\xi_* \ra > {2}\la \xi_* - \xi^-\ra}
+  {\bf 1}_{\la \xi_* \ra < |\xi|/2},
\]
we decompose respectively
\begin{align*}
\cA_{2,1,p}
=
B_1+ B_2 +B_3\,.
\end{align*}
On the sets for
above integrals, we have $\la \xi_* -\xi^- \ra \lesssim \,
\la \xi_* \ra$, because $| \xi^- | \lesssim | \xi_*|$
that follows from  $| \xi^-|^2 \le 2 | \xi_* \cdot\xi ^-| \lesssim |\xi^-|\, | \xi_*|$.
Furthermore, on the sets for $B_1$ and $B_2$  we have $\la \xi \ra \sim \la \xi_* \ra$,
so that
$\la \xi_* -\xi^- \ra \lesssim \ \la \xi \ra$ and $b\,\, {\bf 1}_{ | \xi^- | \ge {1\over 2} \la \xi_*\ra } {\bf 1}_{\la \xi_* \ra \geq |\xi|/2}$
is bounded.
Putting again
 $\hat G(\xi) = \la \xi \ra^{s'} M_\lambda^\delta (\xi) \hat g(\xi)$ and $\hat H(\xi) = \la \xi \ra^{s'}\hat h(\xi)$,
by Lemma \ref{basic-inequality} we have
\begin{align*}
|B_1| ^2
\lesssim& \left[\iiint
\left
|\frac{\hat \Phi_c (\xi_* - \xi^-)}{\la \xi_* - \xi^-\ra^{s'}} \right|^2  | \hat f (\xi_* )|^2  \right.\\
&\quad \times
\left\{M^\delta_\lambda (\xi_*)^2
\left( \frac{{\bf 1}_{ \la \xi -\xi *\ra \lesssim \la \xi_{ *}-\xi^-  \ra}}{
\la \xi-\xi_*\ra^{2(s' +\lambda)} }   +
 \frac{ \delta^{2N_0} {\bf 1}_{ \la \xi -\xi *\ra \lesssim \la \xi_{ *}-\xi^-  \ra}}{
\la \xi-\xi_*\ra^{2(s'+\lambda-N_0) } } \right) \right.\\
&\quad \left. + \left.
\frac{{\bf 1}_{ \la \xi -\xi *\ra \lesssim \la \xi_{ *} -\xi^- \ra}}{
\la \xi-\xi_*\ra^{2s'} } \right \}
d\xi d\xi_* d \sigma \right] \left(\iiint  |    \hat G(\xi - \xi_* )|^2 |{\hat H} (\xi ) |^2 d\sigma d\xi d\xi_*\right) \,.
\end{align*}
Noting  that  $\la \xi_*\ra \sim \la \xi \ra \sim \la \xi^+ \ra \lesssim \la \xi^+ -u\ra + \la u\ra$ with
$u = \xi_* - \xi^-$, and moreover $\la u \ra \lesssim \la \xi_*\ra$,
we see that if  $\lambda \ge 0$ then
\[
M^\delta_\lambda (\xi_*)^2 \lesssim \frac{
\la \xi^+ -u\ra^{2\lambda } + \la u\ra^{2\lambda}}{(1+ \delta \la u \ra)^{2N_0}}\,.
\]
This is true even if $\lambda <0$. Therefore,
if $s'+\lambda <3/2$ we have
\begin{align*}
&|B_1 | ^2
\lesssim  \|f\|_{L^1}^2 \int
{\la u \ra^{-( 6 +2\gamma+2s')}} \\
& \times \Big\{  \int_{\la \xi^+ -u \ra \leq \la u \ra} \frac{( \la
\xi^+ -u\ra^{2s' } + \la u\ra^{2\lambda})}{(1+ \delta \la u
\ra)^{2N_0}} \Big( \frac{1}{ \la \xi^+-u\ra^{2(s' +\lambda)} } +
 \frac{ \delta^{2N_0}}{
\la \xi^+-u\ra^{2(s'+\lambda-N_0) } } \Big) d\xi^+ \\
& +  \int_{\la \xi^+ -u \ra \leq \la u \ra} \frac{d\xi^+}{ \la
\xi^+-u\ra^{2s'} } \Big \} du \,\,
 \|M^\delta _\lambda(D)g\|^2_{H^{s'}} \|h\|_{H^{s'}}^2\\
& \lesssim \|f\|^2_{L^1}  \|M_\lambda^\delta (D)g\|^2_{H^{s'}}
\|h\|_{H^{s'}}^2 \int\frac{du}{\la u \ra^{ 3 +2(\gamma+ 2s')} }\,.
\end{align*}
Here we have used the change of variables
$(\xi, \xi_*) \rightarrow (\xi^+, u)$ whose Jacobian is
\begin{align*}
&\Big|\frac{\partial(\xi^+,u)}{\partial (\xi, \xi_*)}\Big|=\Big|\frac{\partial \xi^+}{\partial \xi} \Big|=\frac{ \Big|I+ \frac{\xi}{|\xi|}\otimes
\sigma\Big|} {8}\\
& =\frac{|1+ \frac{\xi}{|\xi|}\cdot\sigma|}{8}=\frac{\cos^2
(\theta/2)}{4}\ge \frac{1}{8}, \qquad \theta\in [0,\frac{\pi}{2}].  \notag
\end{align*}
If  $s' + \lambda \ge 3/2$, in view of  $\gamma + 2s' >0$ we have
\begin{align*}
&|B_1 | ^2
\lesssim  \int | \hat f(\xi_*)|^2 \left\{
{\la u \ra^{2\lambda-( 6 +2\gamma +2s')}}  \log\la u \ra
\right
\} d\xi_*\,\,
 \|M^\delta _\lambda(D)g\|^2_{H^{s'}} \|h\|_{H^{s'}}^2\\
&\lesssim
\|f\|^2_{H^{(\lambda+s'-3)^+}}  \|M^\delta_\lambda(D)g\|^2_{H^{s'}} \|h\|_{H^{s'}}^2,
\end{align*}
because $\la u\ra \lesssim \la \xi_*\ra$ on the set of the
 integral.

As for $B_2$, we first note that, on the set of the integral,
$\xi^+ = \xi-\xi_* +u$ implies
\[ \frac{\la \xi -\xi_*\ra}{2} \le  \la \xi -\xi_*  \ra - |u| \le \la \xi^+ \ra
\le \la \xi - \xi_*\ra +|u| \lesssim   \la \xi - \xi_*\ra \,,\]
so that
\[(
\enskip M^\delta _\lambda(\xi) \sim \enskip ) \enskip M^\delta_\lambda(\xi^+) \sim
M^\delta_\lambda (\xi-\xi_*)\,.
\]
 and hence we have by the Cauchy-Schwarz inequality
\begin{align*}
|B_2| ^2  \lesssim& \|f\|^2_{L^1} \iiint \frac{
|\hat \Phi_c (\xi_* - \xi^-) |   }
{\la \xi_* -\xi^- \ra^{2s'}}|\hat G(\xi -\xi_*)|^2 d \sigma d\xi d\xi_* \\
& \qquad \times \iiint \frac{
|\hat \Phi_c (\xi_* - \xi^-) |   }
{\la \xi_* -\xi^- \ra^{2s'}} |{ \hat H} (\xi ) |^2 d\sigma d\xi d\xi_*\\
\lesssim &
\|f\|^2_{L^1}  \|M^\delta_\lambda(D)g\|^2_{H^{s'}} \|h\|_{H^{s'}}^2\,,
\end{align*}
because   $\gamma + 2s'>0$.

On the  set of integral for  $B_3$ we recall  $\la \xi \ra \sim \la \xi - \xi_*\ra$ and
\[
|M^\delta_\lambda(\xi) -M^\delta_\lambda(\xi-\xi_*) |\lesssim \frac{\la \xi_*\ra}{\la \xi \ra} M^\delta_\lambda(\xi -\xi_*)\,,
\]
so that
\begin{align*}
|B_3 | ^2  \lesssim& \|f\|_{L^1}^2\iiint b\,\, {\bf 1}_{ | \xi^- | \ge {1\over 2} \la \xi_*\ra }\frac{
|\hat \Phi_c (\xi_* - \xi^-) | \la \xi_* \ra  }
{\la \xi\ra^{2s'+1} } |\hat G(\xi -\xi_*)|^2 d \sigma d\xi d\xi_* \\
& \times \iiint b\,\, {\bf 1}_{ | \xi^- | \ge {1\over 2} \la \xi_*\ra }\frac{
|\hat \Phi_c (\xi_* - \xi^-) | \la \xi_* \ra }
{\la \xi\ra^{2s'+1} } |{\hat H} (\xi ) |^2 d\sigma d\xi d\xi_*\,.
\end{align*}
We use the change of variables  $\xi_*  \rightarrow u= \xi_* -\xi^-$.
Note that $| \xi ^-| \ge {1\over 2} \la u +\xi^-\ra $ implies  $|\xi^-| \geq \la u\ra/\sqrt {10}$,
and that
\[
\la \xi_* \ra \lesssim \la \xi_* - \xi^- \ra + |\xi| \sin \theta/2\,.
\]
Then we have
\begin{align*}
&
\iint  b\,\, {\bf 1}_{ | \xi^- | \ge {1\over 2} \la \xi_*\ra }  \frac{
|\hat \Phi_c (\xi_* - \xi^-) | \la \xi_* \ra}
{\la \xi\ra^{2s'+1 } }d\sigma d\xi_*
\lesssim
\int
\frac{{\bf 1}_{\la u\ra \lesssim |\xi|}}{\la u \ra^{3+\gamma +2s'}}\left( \frac {\la u \ra}
{\la \xi \ra}\right)^{2s'} \\
&\qquad \qquad \times
\Big( \int  b\, {\bf 1}_{ | \xi^- |  \gtrsim \la u \ra } \frac{\la u \ra}
{\la \xi \ra} d\sigma +
\int  b \sin (\theta/2)  {\bf 1}_{ | \xi^- |  \gtrsim \la u \ra } d\sigma
  \Big)du\,,
\end{align*}
from which we also can obtain the desired bound for $B_3$ if
$\gamma +2s'>0$.
In fact, the first integral on the sphere is bounded above by
$\la u \ra^{1-2s}/ \la \xi \ra^{1-2s}$ and the second integral  has the same bound
when $s >1/2$. On the other hand, the second integral is bounded by a constant
when $s<1/2$ and   by $|\log (\la \xi \ra /\la u \ra)|$  when $s=1/2$.
The proof of 1) and 2) of the proposition is then completed.
\end{proof}

The combination of Proposition \ref{commu-molli} and
Proposition \ref{prop2.9_amuxy3} together with its remark yield the following
theorem.

\begin{theo} \label{commu}
Assume that  $0< s<1, \gamma+2s>0$. Let
$0<s'<s$ satisfy $
\gamma+2s'>0,\, 2s'\geq (2s -1)^+$.
If a pair $(N_0, \lambda)$ satisfies \eqref{restriction}
then we have

\noindent
{\em 1)}  If $s' + \lambda <3/2$,   we have
\begin{align*}
\Big
| \Big (M_\lambda^\delta (D_v) \,& Q(f,  g) - Q(f,  M_\lambda^\delta (D_v)\, g)  , h \Big ) \Big |\\
&\lesssim  \| f\|_{L^1_{\gamma^+ +(2s-1)^+}} || M_\lambda^\delta(D_v)g||_{H^{s'}_{\gamma^+ +(2s-1)^+}} \, || h||_{H^{s'}}\,.
\end{align*}
{\em 2)}  If $s'+\lambda \ge 3/2$,
we have
\begin{align*}
\Big
| \Big (M_\lambda^\delta(D_v) &\, Q (f,  g) - Q (f,  M_\lambda^\delta(D_v)\, g)  , h \Big ) \Big |\\
&\lesssim  \Big(  \| f\|_{L^1_{\gamma^+ +(2s-1)^+}} +  \| f\|_{{ H^{ (\lambda +s'-3)^+}}} \Big)|| M_\lambda^\delta(D_v)g||_{H^{s'}_{\gamma^+ +(2s-1)^+}} \, || h||_{H^{s'}}\,.
\end{align*}
Furthermore, if $s>1/2$ and $\gamma > - 1 $ then the same conclusion as above holds
even when the condition   \eqref{restriction} is  replaced by \eqref{slight}.
When $0<s<1/2$ and $\gamma >0$, we can use
$\|M_\lambda^\delta(D_v)g\|_{H^{s'}_{\gamma/2}} $ $\|h\|_{H^{s'}_{\gamma/2}}$
for the corresponding terms in above estimates with smaller weight in the
variable $v$.
\end{theo}

We recall also the following upper bound estimate, Proposition 2.1 of \cite{amuxy4-3}, where we need the
assumption $\gamma+2s >0$ (see also Theorem 2.1 from \cite{amuxy3}).
\begin{prop}\label{hard-upper}
Let $\gamma+2s >0$ and $0<s<1$. For any $r \in [2s-1,2s]$ and $\ell \in [0, \gamma +2s]$ we have
\begin{align*}
\Big| \Big( Q(f,g),h \Big)_{L^2(\RR^3)} \Big | \lesssim
\|f\|_{L^1_{\gamma+2s} }\|g\|_{H^r_{\gamma+2s -\ell}} \|h\|_{H^{2s-r}_\ell}\,.
\end{align*}
\end{prop}

In the following analysis, we shall need an interpolation inequality
concerning weighted type Sobolev spaces in
$v$, see for instance \cite{desv-wen1, HMUY}.

\begin{lemm}\label{lemm2.3}
For any $k\in\RR, p\in\RR_+, \delta>0$,
\begin{equation*}\label{interp-1}
\|f\|^2_{H^{k}_p(\RR^3_v)}\leq C_\delta \|f\|_{H^{k-\delta}_{2
p}(\RR^3_v)} \| f\|_{H^{k+\delta}_0(\RR^3_v)}.
\end{equation*}
\end{lemm}

And also another interpolation in $L^q$ is given by

\begin{lemm}\label{interp-2-lem}
Let $1 < q <p $.  Assume that
$f \in L^p$ and $\la v \ra^\ell f \in L^1$ for any $\ell$.
Then $\la v \ra^\ell \, f\in L^q$ for any $\ell$.
More precisely, we have
\begin{equation*}\label{interp-2}
\|f\|_{L^q_\ell} \le 2 \|f\|_{L^p}^{\frac{p(q-1)}{q(p-1)}} \|f\|_{L^1_{\frac{\ell q(p-1)}{p-q}}}^{\frac{p-q}{q(p-1)}}\,.
\end{equation*}
\end{lemm}

\begin{proof} Take $\lambda>0$, we rewrite
\begin{align*}
\|f\|_{L^q_\ell}^q & = \int_{\la v \ra^{\ell q} |f(v)|^{q-p} \le \lambda} \la v \ra^{\ell q} |f(v)|^{q} dv +
\int_{\la v \ra^{\ell q} |f(v)|^{q-p} > \lambda}\la v \ra^{\ell q} |f(v)|^{q} dv \\
& \le \lambda \|f\|_{L^p}^p + \lambda^{\frac{q-1}{q-p}}\|f\|_{L^1_{\frac{\ell q(p-1)}{p-q}}}\,.
\end{align*}
Taking
$$
\lambda = \|f\|^{\frac{(p-q)}{(p-1)}}_{L^1_{\frac{\ell q(p-1)}{p-q}}}\,\, \|f\|_{L^p}^{-\frac{p(p-q)}{(p-1)}},
$$
we obtain the desired estimate.
\end{proof}

\section{Smoothing effect of $L^2$ weak solutions} \label{section4}
\setcounter{equation}{0}

We start from a weak solution in $L^2$ with bounded moments.

\begin{theo}\label{L2-theo}
Assume that $0<s<1,\,\gamma+2s >0$. If $f$ belongs to $L^\infty([t_0, T]; L^2_\ell(\RR^3))$ for any $\ell\in\NN$ and is a non-negative weak solution of \eqref{1.1} ,  then for any $t_0<\tilde{t}_0<T$, we have
$$
f\in L^\infty([\tilde{t}_0, T]; \cS(\RR^3)).
$$
\end{theo}

\begin{proof} Without loss of generality, take $t_0=0$.
Assume that, for $a\geq 0$, we have
\begin{align}\label{induction-hypothesis-1}
\sup_{[0,T]} \|f(t,\cdot)\|_{H^{a}_\ell} < \infty\, \enskip \mbox{for any $\ell\in\NN$}.
\end{align}
Take  $\lambda(t) = N t +  a$ for $N>0$. Choose $N_0 = a + (5+\gamma)/2$. Then the pair
$(N_0, \lambda(t))$ satisfies \eqref{restriction}.
If we choose  $N, T_1 >0$ such that $NT_1=(1-s)$, then
\begin{equation*}\label{caution}
\lambda(T_1) - N_0 -a \le   \lambda(T_1) - N_0  <   -3/2\,,
\end{equation*}
from which we  have, for $t, t' \in [0, T_1]$,
\begin{equation}\label{str}
 M_{\lambda(t)}^\delta f(t')
\in L^\infty([0,T_1]  \times [0,T_1] ; H^{3/2}_\ell(\RR^3)\cap L^\infty(\RR^3)),
\end{equation}
because of \eqref{induction-hypothesis-1}. By the same way as in (3.5) and (3.6) of
\cite{MUXY-DCDS}, we have
\begin{equation}\label{continuity}
 M_{\lambda(t)}^\delta f(t) \in C([0,T_1]; L^2(\RR^3)),
\end{equation}
and for any $t\in ]0, T_1]$, we have
\begin{eqnarray}\label{energyequality1}
&&\frac{1}{2}\int_{\RR^3} \big(M_{\lambda(t)}^\delta f(t)\big)^2 dv
-\frac{1}{2}\int^t_0 \int_{\RR^3}
f(\tau)\Big(\partial_\tau  (M_{\lambda(\tau)}^\delta)^2\Big)f(\tau) dv d\tau \notag \\
&& = \frac{1}{2}\int_{\RR^3} \big(M_{\lambda(0)}^\delta f_0\big)^2  dv\\
&& + \int^t_0
\Big( Q\big (f(\tau), M_{\lambda(\tau)}^{\delta}f(\tau)\big ), \,
M_{\lambda(\tau)}^{\delta} f(\tau)\Big)_{L^2} d\tau \notag\\
&& + \int^t_0
\Big( M_{\lambda(\tau)}^{\delta}Q\big (f(\tau), f(\tau)\big )
-Q\big (f(\tau), M_{\lambda(\tau)}^{\delta}f(\tau)\big ), \,
M_{\lambda(\tau)}^{\delta} f(\tau)\Big)_{L^2} d\tau, \notag
\end{eqnarray}
by taking $(M_{\lambda(t)}^\delta)^2 f(t)$ as a test function in the definition of  the weak solution,
though it does not belong to $L^\infty([0,T_1];W^{2,\infty}(\RR^3))$.
In fact, we can show
\eqref{continuity} and  \eqref{energyequality1}
under a weaker condition than \eqref{str}, which will be given in Lemma
\ref{conti-w-equation} below.

Noting
\[
\partial_t M_{\lambda(t)}^\delta = N (\log \la \xi \ra )M_{\lambda(t)}^\delta\,,
\]
by Theorem \ref{commu} we have
\begin{align}\label{imp-ene}
&\frac{1}{2}\|\big (M_\lambda^\delta f \big)(t)\|_{L^2}^2
\leq  \frac{1}{2}\|f(0)\|_{H^a}^2 + \int_0^t
\Big( Q(f(\tau), \big(M_\lambda^\delta f\big)(\tau)\,) ,  \big(M_\lambda^\delta f\big)(\tau)\Big) d\tau
\notag \\
&\qquad +C_f \int_0^t
\|\big(
M_\lambda^\delta f\big)(\tau)\|_{H^{s'} _{(2s+\gamma -1)^+}}
\|\big(M_\lambda^\delta f\big)(\tau) \|_{H^{s'}} d\tau\\
&\qquad  + C N  \int_0^t \|(\log \la D\ra)^{1/2} \big (M_\lambda^\delta f \big)(\tau)\|_{L^2}^2 d\tau\,.
\notag
\end{align}
Since the uniform coercive estimate  \eqref{coer-uni} together with the interpolation in
the Sobolev space yields
\[
\Big( Q(f(\tau), \big(M_\lambda^\delta f\big)(\tau)\,) ,  \big(M_\lambda^\delta f\big)(\tau)\Big)
\le
-c_f \|\big(M_\lambda^\delta f\big)(\tau)\|^2_{H^s_{\gamma/2}}
+ C_f\|f(\tau)\|^2_{H_{\gamma/2}^{-2}}\,,
\]
by means of Lemma \ref{lemm2.3}  we have
\begin{equation}\label{ene-34}
\|\big (M_\lambda^\delta f \big)(t)\|_{L^2}^2
+ c_f \int_0^t \|\big(M_\lambda^\delta f\big)(\tau)\|^2_{H^s_{\gamma^+/2}}d\tau
\leq  \|f(0)\|_{H^a}^2 + C_f \int_0^t \|f(\tau)\|^2_{H_{\ell}^{a}} d\tau\,.
\end{equation}
Taking $\delta \rightarrow +0$ and $t=T_1$, we have $f(T_1) \in H^{\lambda(T_1)}= H^{NT_1+a}$.
This is true for any $0<T_1\leq T$. Choosing $N=(1-s)T^{-1}_1$, we have that for any $0<T_1\leq T$,
$$
f(T_1) \in H^{(1-s)+a}\,.
$$
Fix $0<s_0<(1-s)$. Then, by using  Lemma \ref{lemm2.3} and assumption \eqref{induction-hypothesis-1}, we see that for any $0<t_1< \tilde{t}_0$ and any $\ell$,
$$
\sup_{[t_1, T]} \|f(t,\cdot)\|_{H^{s_0+a}_\ell} < \infty\,.
$$

We can restart by replacing $a$ by $a + s_0=a_1$ and
$t_0$ by $t_1$. By induction, for $a_0=0, a_k=k\,s_0$, and $t_k=\tilde{t}_0-(2k)^{-1}(\tilde{t}_0-t_0)$, we have
for any $k\in\NN$ and any $\ell$,
$$
f\in L^\infty([t_k,\, T]; H^{a_k}_\ell(\RR^3)),
$$
which concludes the proof of Theorem \ref{L2-theo}.
\end{proof}

\begin{rema}\label{pf-thm1-2-1}
When $0<s<1/2$ and $\gamma > 0$ we can use
$\int_0^t \|\big(M_\lambda^\delta f\big)(\tau) \|^2_{H^{s'}_{\gamma/2}}d\tau$
for  the corresponding term
in  \eqref{imp-ene}. Hence, instead of \eqref{ene-34}, we can obtain
\begin{equation*}
\|\big (M_\lambda^\delta f \big)(t)\|_{L^2}^2
\leq  \|f(0)\|_{H^a}^2 + C_f \int_0^t \|f(\tau)\|^2_{H_{\gamma/2}^{-2}} d\tau\,,
\end{equation*}
which shows that $f(t) \in L^\infty([0,T];  L^2 \cap L^1_{2}(\RR^3))$ implies $f(t) \in H^\infty(\RR^3)$ for $t>0$.
\end{rema}

\begin{lemm}\label{conti-w-equation}
Let $T_1 >0$ and
let $M^\delta_{\lambda(t)}(\xi)$ be defined by  \eqref{mollifier} with $\lambda = \lambda(t) = Nt +a$ for $NT_1 <1$ and $a \in \RR$.
Suppose that
$$
f \in L^1([0,T_1]; L^1_{\max\{\gamma+2s, 2\}}(\RR^3) )\cap L^\infty([0,T_1];
H^a(\RR^3)).
$$
If there exists  { $s_1 >s$ } such that
$$
M_{\lambda(t)}^\delta f(t', v) \in L^\infty([0,T_1]_t \times
[0,T_1]_{t'}; H^{ s_1}_{\ell_0}(\RR^3_v))
$$
for $\ell_0 = \max\{\gamma/2 +s, \gamma^+ + (2s-1)^+ \}$, then we
have \eqref{continuity}, and \eqref{energyequality1} for any $t \in
[0,T_1]$.
 Furthermore, if $0<s<1/2$ and $\gamma >0$ we can take $\ell_0 = \gamma/2 +s$.
\end{lemm}
\begin{proof}
In Definition
\ref{defi1}, taking $\varphi(t,v) = \psi(v) \in
C_0^{\infty}(\RR^3)$,  we get
\[\int_{\RR^3} f(t) \psi dv - \int_{\RR^3} f(t') \psi dv = \int_{t'}^t d \tau \int_{\RR^3} Q(f(\tau),f(\tau))
\psi dv\,, \enskip 0 \leq t' \leq t \leq T_0\,.\]
By taking a sequence $\{\psi_j(v)\}_{j=1}^\infty \subset C_0^\infty(\RR^3_v)$ such that $(M_{\lambda(t)}^\delta)^{-1} \psi_j \rightarrow M_{\lambda(t)}^\delta f(t)$
in $H_{\ell_0}^s$,
we can  set  $\psi
=  (M_{\lambda(t)}^\delta)^2 f(t)$ for a fixed $t$
because
\[
 \left | \int_{\RR^3} f(t')  (M_{\lambda(t)}^\delta)^2 f(t) dv \right| \le \|M_{\lambda(t)}^\delta f(t')\|_{L^2}\|M_{\lambda(t)}^\delta f(t)\| _{L^2}< \infty\,,
 \]
and by noting
\[
(Q(f,f), (M_\lambda^\delta)^2f)
= (Q(f, M_\lambda^\delta f), M_\lambda^\delta f)  + ( M_\lambda^\delta Q(f,f) - Q(f,M_\lambda^\delta f), M_\lambda^\delta f)\,,
\]
we have
\begin{align*}
&\left| \int_{t'}^t d \tau \int_{\RR^3} Q(f(\tau),f(\tau))
 (M_{\lambda(t)}^\delta)^2 f(t) dv \right|\\
&\lesssim   \int_{t'}^t \|f(\tau)\|_{L^1_{\gamma+2s}}d \tau  \Big(\sup_{\tau, t \in [0,T_1] }   \|M_{\lambda(t)}^\delta f(\tau)\|_{H^s_{\gamma/2 +s} }\|M_{\lambda(t)}^\delta f(t)\|_{H^s_{\gamma/2 +s} }\Big)
\\
&\quad  +  \Big( \int_{t'}^t \|f(\tau)\|_{L^1_{\gamma^+ +(2s-1)^+}} d\tau+   |t-t'|\sup_{\tau\in [0,T_1] }\|f(\tau)\|_{H^a} \Big)\\
& \qquad \qquad \times \Big(\sup_{\tau, t \in [0,T_1] }\|M_{\lambda(t)}^\delta f(\tau)\|_{H^s_{\gamma^+ +(2s-1)^+}}
\|M_{\lambda(t)}^\delta f(t)\|_{H^s} \Big)\,,
\end{align*}
thanks to Proposition \ref{hard-upper}  and Theorem \ref{commu}.
Setting  $\psi
=   (M_{\lambda(t')}^\delta)^2 f(t')$ also and taking the sum, we
obtain
\begin{align}\label{difference}
&\int_{\RR^3} \big( M_{\lambda(t)}^\delta f(t)\big)^2 dv- \int_{\RR^3}
\big(M_{\lambda(t')}^\delta f(t') \big)^2 dv \notag \\
&\qquad =  \int_{\RR^3} f(t) \left(  (M_{\lambda(t)}^\delta)^2  -(M_{\lambda(t')}^\delta)^2
\right) f(t') dv \\
& \qquad \qquad + \int_{t'}^t d \tau \int_{\RR^3} Q(f(\tau),f(\tau))
\left(  (M_{\lambda(t)}^\delta)^2 f(t) +  (M_{\lambda(t')}^\delta)^2 f(t')   \right)dv\,. \notag
\end{align}
Since it follows from the mean value theorem that
the first term on the right hand side is estimated by
\[  |t-t'|  \sup_{ 0\le t' < \tilde \tau <t \le T_1}
\|M_{\lambda(\tilde \tau)}^\delta f(t)\|_{L^2}\| (\log \la D \ra )
M_{\lambda(\tilde \tau)}^\delta f(t')\|_{L^2},
\]
 we obtain (\ref{continuity}), namely $M_{\lambda(t)}^\delta f(t)
\in C([0,T_0]; L^2(\RR^3))$.

Taking  $\psi =   (\log\la D \ra)^2(M_{\lambda(t')}^\delta)^2
f(t')$, we also have
 $$
 (\log \la D \ra ) M_{\lambda(t)}^\delta f(t)
\in C([0,T_0]; L^2(\RR^3))\,.
$$
Taking the difference, instead of  \eqref{difference}, we get
\begin{align*}
&\int_{\RR^3} \big( M_{\lambda(t)}^\delta f(t)\big)^2 dv+ \int_{\RR^3}
\big(M_{\lambda(t')}^\delta f(t') \big)^2 dv \notag \\
&\qquad =  \int_{\RR^3} f(t) \left(  (M_{\lambda(t)}^\delta)^2  +(M_{\lambda(t')}^\delta)^2
\right) f(t') dv \\
& \qquad \qquad + \int_{t'}^t d \tau \int_{\RR^3} Q(f(\tau),f(\tau))
\left(  (M_{\lambda(t)}^\delta)^2 f(t) -  (M_{\lambda(t')}^\delta)^2 f(t')   \right)dv\,, \notag
\end{align*}
which shows
\[
\lim_{t' \rightarrow t}
\int_{\RR^3} f(t) \left(  (M_{\lambda(t)}^\delta)^2  +(M_{\lambda(t')}^\delta)^2
\right) f(t') dv = 2 \int_{\RR^3}  \big(M_{\lambda(t)}^\delta f(t)\big )^2dv\,,
\]
and moreover
\begin{equation}\label{product}
\lim_{t' \rightarrow t}
\int_{\RR^3} \big (M_{\lambda(t)}^\delta f(t)\big) \big (M_{\lambda(t')}^\delta
 f(t')\big) dv =  \int_{\RR^3} \big(M_{\lambda(t)}^\delta f(t)\big )^2dv \,.
\end{equation}

To prove \eqref{energyequality1} we introduce
\[
M_{\lambda(t)}^{\delta,\kappa} = \frac{M_{\lambda(t)}^\delta(\xi)}{1+\kappa \la \xi \ra} \,,
\]
with a new parameter $\kappa >0$. Divide $[0,t]$ into
$k$ subintervals with the same length and put $t_j = jt/k$ for $j=0,\cdots, k$.
Similar to \eqref{difference},  we have
\begin{align}\label{difference-j}
&\int_{\RR^3} \big( M_{\lambda(t)}^{\delta,\kappa} f(t_j)\big)^2 dv- \int_{\RR^3}
\big(M_{\lambda(t_{j-1})}^{\delta,\kappa} f(t_{j-1}) \big)^2 dv \notag \\
&\qquad =  \int_{\RR^3} f(t_j) \left(  (M_{\lambda(t_j)}^{\delta,\kappa})^2
-(M_{\lambda(t_{j-1})}^{\delta,\kappa})^2
\right) f(t_{j-1}) dv \\
& \qquad \qquad + \int_{t_{j-1}}^{t_j} d \tau \int_{\RR^3} Q(f(\tau),f(\tau))
\left(  (M_{\lambda(t_j)}^{\delta,\kappa})^2 f(t_j) +  (M_{\lambda(t_{j-1})}^{\delta,\kappa})^2 f(t_{j-1})   \right)dv\,.
\notag
\end{align}
Since  we have
\begin{align*}
& \int f(t_j) \left(  (M_{\lambda(t_j)}^{\delta,\kappa})^2
-(M_{\lambda(t_{j-1})}^{\delta,\kappa})^2
\right) f(t_{j-1}) dv\\
&=\int 2N  f(t_j)(\log \la D \ra)(M_{\lambda(\tau_j)}^{\delta,\kappa})^2
f(t_{j-1}) dv(t_j-t_{j-1})\enskip \mbox{ $\tau_j \in ]t_{j-1},t_j[$}\\
&= 2 N \int
 \big( \sqrt {\log\la D \ra}M_{\lambda(t_j)}^{\delta,\kappa}f(t_j) \big)
\big (\sqrt {\log\la D \ra}M_{\lambda(t_{j-1})}^{\delta,\kappa}f(t_{j-1}) \big ) dv(t_j-t_{j-1})\\
&+ N^2 \Big(\sup_{\tau', \tau''\in [0,T_1]}
\|\log \la D \ra M_{\lambda(\tau')}^{\delta,\kappa} f(\tau'')\|_{L^2}
\Big)^2  O\big( |t_j -t_{j-1}|^2 \big ) \,,
\end{align*}
it follows from a similar formula as \eqref{product} that
\begin{align*}
&\lim_{k \rightarrow \infty} \sum_{j=1}^k
\int f(t_j) \left(  (M_{\lambda(t_j)}^{\delta,\kappa})^2
-(M_{\lambda(t_{j-1})}^{\delta,\kappa})^2
\right) f(t_{j-1}) dv\\
&\qquad = N \int_0^t \int_{\RR^3} \big (\sqrt {\log\la D \ra} M_{\lambda(\tau)}^{\delta,\kappa}f(\tau) \big)^2dv d\tau\,.
\end{align*}
Summing up \eqref{difference-j} with respect to $j=1,\cdots, k$ and making $k \rightarrow \infty$,
we obtain
\begin{eqnarray}\label{energyequality-kappa}
&&\frac{1}{2}\int_{\RR^3} \big(M_{\lambda(t)}^{\delta,\kappa} f(t)\big)^2 dv
-\frac{1}{2}\int^t_0 \int_{\RR^3}
f(\tau)\Big(\partial_\tau  (M_{\lambda(\tau)}^{\delta,\kappa})^2\Big)f(\tau) dv d\tau \notag \\
&& = \frac{1}{2}\int_{\RR^3} \big(M_{\lambda(0)}^{\delta,\kappa} f_0\big)^2  dv \\
&&+ \int^t_0
\Big( Q\big (f(\tau), M_{\lambda(\tau)}^{\delta}f(\tau)\big ), \,
\frac{M_{\lambda(\tau)}^{\delta}}{(1+\kappa\la D \ra)^2} f(\tau)\Big)_{L^2} d\tau \notag\\
&&+ \int^t_0
\Big( M_{\lambda(\tau)}^{\delta}Q\big (f(\tau), f(\tau)\big )-Q\big (f(\tau), M_{\lambda(\tau)}^{\delta}f(\tau)\big ), \,
\frac{M_{\lambda(\tau)}^{\delta}}{(1+\kappa\la D \ra)^2} f(\tau)\Big)_{L^2} d\tau, \notag
\end{eqnarray}
thanks to Proposition \ref{hard-upper}  and Theorem \ref{commu}.
In fact, for example, we have
\begin{align*}
&\left| \Big(
Q\big (f(\tau), M_{\lambda(t_j)}^{\delta}f(\tau)\big ), \,
\frac{M_{\lambda(t_j)}^{\delta}}{(1+\kappa\la D \ra)^2} f(t_j)\Big)_{L^2}\right|\\
&\lesssim \sup_{\tau, t_j \in [0,T_1] } \Big \{  
\|f(\tau)\|_{L^1_{\gamma+2s}} \|M_{\lambda(t_j)}^\delta f(\tau)\|_{H^s_{\gamma/2 +s} }\|\frac{M_{\lambda(t_j)}^{\delta}}{(1+\kappa\la D \ra)^2} f(t_j)\|_{H^s_{\gamma/2 +s} }
\Big \}\,,
\end{align*}
and hence the Lebesgue convergence theorem yields \eqref{energyequality-kappa}
because,
\[
\|\frac{M_{\lambda(t_j)}^{\delta}}{(1+\kappa\la D \ra)^2} f(t_j)\|_{H^s_{\gamma/2 +s} }
\rightarrow
\|\frac{M_{\lambda(\tau)}^{\delta}}{(1+\kappa\la D \ra)^2} f(\tau)\|_{H^s_{\gamma/2 +s} }
\enskip  \mbox{as $|t_j-\tau|\rightarrow 0$}.
\]
Taking $\kappa \rightarrow 0$ in \eqref{energyequality-kappa} we obtain the desired formula.
The last assertion of the lemma follows easily from the one of Theorem \ref{commu}.
\end{proof}
\smallskip

\section{Smoothing effect of $L^1$ weak solutions} \label{section5}
\setcounter{equation}{0}

We come back to the proof of
Theorem \ref{Theorem1} starting from the  $L^1$ weak solution. The fist part of the theorem is restated as follows:

\begin{theo}\label{L1-theo1}
Assume that $0<s < 1,\,\gamma>  \max \{-2s,-1\}$. If $f$ belongs to $L^\infty([t_0, T]; L^1_\ell(\RR^3))$  for any $\ell\in\NN$ and is a weak solution of \eqref{1.1},
then for any $t_0<\tilde{t}_0<T$, we have
$$
f\in L^\infty([\tilde{t}_0,\, T]; \cS(\RR^3)).
$$
\end{theo}
\begin{proof}
By Theorem \ref{L2-theo}, it is sufficient to prove,
for any $0<t_1\leq T$, (take again $t_0=0$)
\begin{equation}\label{l-2}
f\in L^\infty([t_1, T]; L^2_\ell(\RR^3)).
\end{equation}
Since $L^1(\RR^3)\subset H^{-3/2-\varepsilon}$, we assume that for any $\ell$
and any
$0<\varepsilon<\!<1$
\begin{equation}\label{first-hypo}
\sup_{[0,T]} \|f(t,\cdot)\|_{H^{-3/2- \varepsilon}_\ell} < \infty \,.
\end{equation}

 As in the proof of Theorem \ref{L2-theo},  we shall prove the theorem by  induction. Assume that for $0> a \ge -3/2 - \varepsilon$, we have
$$
\sup_{[0,T]} \|f(t,\cdot)\|_{H^{a}_\ell} < \infty\,.
$$
Take also $\lambda(t) = N t +  a$ for $N >0$.

\smallskip
We first consider the case $0<s \le 1/2$.  Choose $N_0 = a +(5+\gamma)/2\ge 1- \varepsilon + (\gamma/2) >0$
such that \eqref{restriction}
is fulfilled. Put $\varepsilon_0 = (1- 2s')/8 >0$ and consider $\varepsilon = \varepsilon_0$,
where $0< s' < s$ is chosen to satisfy $\gamma +2s' >0$.
If we choose $N, T_1>0$ such that $NT_1 = \varepsilon_0$ then
\begin{equation*}\label{caution2}
s+ \lambda(T_1) - N_0 -a = s + \varepsilon_0 -N_0 \le  s-1 + 2 \varepsilon_0-(\gamma/2)
<  (s'-1/2) + 2 \varepsilon_0 <0\,,
\end{equation*}
which shows
\begin{equation}\label{H-S-est}
M_{\lambda(t)}^\delta f(t)  \in L^\infty([0, T_1];  H^s_\ell (\RR^3))\,.
\end{equation}
This and
Lemma \ref{conti-w-equation}  lead to  (\ref{energyequality1}), and hence
we obtain \eqref{imp-ene} using Theorem \ref{commu}, and \eqref{ene-34}
by means of \eqref{coer-uni} and Lemma \ref{lemm2.3}.  The same procedure as in the proof
of Theorem \ref{L2-theo} shows  \eqref{l-2} by induction.

\smallskip
When $s >1/2$ we choose $1/2 < s' < s$ such that $\gamma +2s' >0, 2s' \ge (2s-1) $.
 Choose $N_0 = a + (5+\gamma+2s'-1)/2$ such that
\eqref{slight} is satisfied. Put $\varepsilon_0 = (\gamma +1 )/10 >0$ and consider $\varepsilon = \varepsilon_0$.
Then, we have
\begin{equation}\label{caution3}
s+ \lambda(T_1) - N_0 -a = s + \varepsilon_0 -N_0 \le  s-s' + 2 \varepsilon_0-(1+\gamma)/2
= s-s'-3 \varepsilon_0 \,. 
\end{equation}
Since we may assume $s-s' \le \varepsilon_0$, \eqref{caution3} also shows \eqref{H-S-est}, 
which completes the proof of the theorem by the same way as in the case $0<s \le 1/2$.
\end{proof}

In view of Remark \ref{pf-thm1-2-1} and the last assertion of Lemma \ref{conti-w-equation},
the proof of Theorem \ref{L1-theo1}
in the case $0<s < 1/2$ leads us easily to the following theorem where the assumption
\eqref{moment} can be removed.

\begin{theo}\label{H-infty-smooth}
Suppose that the cross section $B$ of the form \eqref{part2-hyp-2} satisfies  \eqref{part2-hyp-phys} and \eqref{angular}
with $0< s< 1/2$ and $\gamma >0$.  If
\[
f \in L^\infty([0,T]; L^1_{\max\{2, \gamma/2 +s\}}(\RR^3) \cap L \log L) \cap L^1([0,T]; L^1_{2+ \gamma}(\RR^3))
\]
is a weak solution, then $f  \in L^\infty([t_0,T]; \, H^\infty(\RR^3))$
for any $t_0 \in ]0,T[$.
\end{theo}

We consider now the second part of Theorem \ref{Theorem1}, which is
stated as follows: 
\begin{theo}\label{L1-5}
Assume that $-1\geq \gamma>-2s$. Let $f\in L^\infty([t_0, T]; L^1_\ell(\RR^3))$ for any $\ell\in\NN$ be a  weak solution of \eqref{1.1} satisfying
the entropy dissipation estimate
$$
\int^{T}_{t_0} D(f(t), f(t))dt<+\infty.
$$
Then for any $t_0<\tilde{t}_0<T$, we have
$$
f\in L^\infty([\tilde{t}_0,\, T]; \cS(\RR^3)).
$$
\end{theo}

For the proof, we  only need to reconsider the term $A_1$ defined in \eqref{A_1} under the hypothesis  $-1 \ge \gamma >-2s$.
Note that we can now choose arbitrarily large $N_0$ in \eqref{mollifier} because neither
\eqref{restriction} nor \eqref{slight} is required. 
Hence  $(M^\delta_{\lambda(t)})^2 f(t)$ belongs to
$W^{2,\infty}$, which enable us to take  $(M^\delta_{\lambda(t)})^2 f(t)$ as a test function.
However $\lambda(t)$ can not be taken  as large as we want, because it is also restricted to
the small gain regularity coming from the dissipation estimate. Thanks to Theorem \ref{L2-theo}, it suffices to show $f \in L^\infty([T_0,T_1]; L^2_\ell)$ by induction,
starting from \eqref{first-hypo}  where we take again $t_0 =0$.

It follows from  \eqref{M-inequality2} of Lemma \ref{basic-inequality}, \eqref{later-use1} and \eqref{later-use2} that $A_1$ can be replaced by
\begin{align}\label{tilde-A1}
\tilde{A}_{1, \lambda} &=\iint_{\RR^6}\frac{|\hat{f}(\xi_*)|}
{\la\xi_*\ra^{3+\gamma}}|\hat{g}(\xi-\xi_*)|\,
|\hat{h}(\xi)|\la\xi\ra^{\lambda}
{\bf 1}_{\la \xi_* \ra \geq \sqrt{2}|\xi|}  \frac{\la \xi\ra}{\la \xi_*\ra}d\xi d\xi_*
\end{align}
We divide the proof in three steps.

\smallskip
\noindent
{\bf 1st step:}
Take $s' >1/2$ such that $\gamma +2s' >0$ and $s'<s$. Put $s_0 = \frac1 4 (\gamma+2s')$.
For arbitrary $t>0$ and $N >0$ satisfying $Nt = s_0$, we set
\[
\lambda_1(\tau) = N \tau -\frac 3 2 -\varepsilon \enskip \mbox{for} \enskip \tau \in [0, t]\,,
\]
where $\varepsilon >0$ is arbitrarily small.
If we substitute $\lambda = \lambda_1(\tau)$ into \eqref{tilde-A1} then, in view of $N \tau \le s_0$, we have
\begin{align*}
\tilde A_{1.\lambda_1}(\tau)
&\lesssim   \|\hat{g}\|_{L^\infty}\iint_{\RR^6}\frac{|\hat{f}(\xi_*)|}
{\la\xi_*\ra^{3+\gamma-s_0}}\frac{
|\hat{h}(\xi)|}{\la\xi\ra^{3/2+\varepsilon}}
d\xi d\xi_*\\
&\lesssim  \|\hat f\|_{L^{3/(2s')}}
\|g\|_{L^1}\|h\|_{L^2} \lesssim\|f\|_{L^{3/(3-2s')}} \|g\|_{L^1} \|h\|_{L^2}
\end{align*}
because of the H\"older inequality and the fact that $(3+\gamma -s_0)\{3/(3-2s')\} >3$.
By means of Lemma \ref{interp-2-lem}, we have for some $\ell_0 >0$
\begin{align*}
\tilde A_{1,\lambda_1} &\lesssim \Big( \|\la v \ra^{\gamma} f\|_{L^{3/(3-2s)}}  + \|f\|_{L^1_{\ell_0}}\Big) \|g\|_{L^1} \|h\|_{L^2} \\
&\lesssim  \Big( \|\la v \ra^{\gamma/2}\sqrt{f}\|^2_{H^s}  + \|f\|_{L^1_{\ell_0}}\Big) \|g\|_{L^1} \|h\|_{L^2} \,.
\end{align*}
Putting $f =g = f(\tau, v)$ and $h = M_{\lambda_1(\tau)}^\delta f(\tau, v)$, we have a term coming from $\tilde A_{1,\lambda_1}$
in estimating
\[
\int_0^t
\Big(   M_{\lambda_1}^\delta  Q(f(\tau),  f(\tau))  -Q(f(\tau), M_{\lambda_1}^\delta f(\tau)), M_{\lambda_1}^\delta f(\tau) \Big)  d\tau
\]
as follows:
\begin{align*}
&\Big(\sup_{\tau \in [0,t] }      \|f(\tau)\|_{L^1} \|M_{\lambda_1(\tau)}^\delta f(\tau)\|_{L^2}
\Big) \int_0^t   \|\la v \ra^{\gamma/2}\sqrt{f(\tau) }\|^2_{H^s} d\tau\\
&+ \Big( \sup_{\tau \in [0,t] }      \|f(\tau)\|^2_{L^1_{\ell_0}}\Big)\sqrt t
\Big(\int_0^t \|M_{\lambda_1(\tau)}^\delta f(\tau)\|^2_{L^2} d\tau\Big)^{1/2}\\
& \le \frac{1}{10}\sup_{\tau \in [0,t] }\|M_{\lambda_1(\tau)}^\delta f(\tau)\|^2_{L^2} + C_f
\Big\{
\Big(\int_0^t D(f(\tau), f(\tau)) d\tau\Big)^2\\
& \qquad \qquad \qquad +t \sup_{\tau \in [0,t] }      \|f(\tau)\|^4_{L^1_{\ell_0}}+
\int_0^t \|M_{\lambda_1(\tau)}^\delta f(\tau)\|^2_{L^2} d\tau
\Big\} \,,
\end{align*}
where we have used Corollary \ref{entropy-dissipation-coer} in the last inequality.
Instead of \eqref{ene-34}, we obtain
\begin{align*}
&\|M_{\lambda_1(t)}^\delta f(t)\|^2_{L^2} -\frac{1}{10}
\sup_{\tau \in [0,t] }\|M_{\lambda_1(\tau)}^\delta f(\tau)\|^2_{L^2}
+ \int_0^t \|M_{\lambda_1(\tau)}^\delta f(\tau)\|^2_{H^{s'}} d\tau\\
&\quad \lesssim \|M_{\lambda_1(0)}^\delta f(0)\|^2_{L^2} + \Big(\int_0^t D(f(\tau), f(\tau)) d\tau\Big)^2
\\
& \qquad \qquad \qquad + t  \sup_{\tau \in [0,t] }      \|f(\tau)\|^4_{L^1_{\ell_0}} { + \int_0^t \|f(\tau)\|^2_{H^a_\ell} d\tau}\,. 
\end{align*}
If we consider $\tau \in [0,t]$ instead of $t$ then
the first term on the right hand side can be replaced by $\sup_{\tau \in [0,t]} \|M_{\lambda_1(\tau)}^\delta f(\tau)\|^2_{L^2} $, which absorbs the second term on the right hand side.
Letting $\delta \rightarrow 0$ we obtain, in view of $Nt =s_0$,
\begin{align}\label{each-t}
\|\la D \ra^{s_0 -3/2 -\varepsilon} f(t)\|_{L^2} < \infty,
\end{align}
and
\begin{align}\label{integral-est}
\int_0^t \|\la D \ra^{N\tau -3/2 -\varepsilon} f(\tau)\|^2_{H^{s'}} d\tau < \infty.
\end{align}

\smallskip
\noindent
{\bf 2nd step:}  Let $\kappa >0$ be  small arbitrarily. Considering $\tau \in [\kappa, t]$ instead of $t$ in \eqref{each-t},
we may assume
\begin{align*}
\sup_{\tau \in [\kappa, t] }\|\la D \ra^{s_0 -3/2 -\varepsilon} f(\tau)\|_{L^2} < \infty\,.
\end{align*}
For arbitrary $t>\kappa$ and $N >0$ satisfying $N(t-\kappa) = s_0$ we set
\[
\lambda_2(\tau) = s_0 +N (\tau-\kappa) -\frac 3 2 -\varepsilon \enskip \mbox{for} \enskip \tau \in [\kappa, t]\,.
\]
If we substitute $\lambda = \lambda_2(\tau)$ into \eqref{tilde-A1} then we have
\begin{align*}
&\tilde A_{1.\lambda_2}(\tau)
\lesssim  \int_{\RR^3}
\frac{
|\la \xi \ra^{s'}\hat{h}(\xi)|}{\la\xi\ra^{3/2+\varepsilon}}\\
&
\Big(\int_{\RR^3}
\frac{  \la \xi_* \ra^{s_0-\frac{3}{2} -\varepsilon} |\hat{f}(\xi_*) |      \la \xi_*-\xi\ra^{N(\tau-\kappa)-\frac{3}{2}-
\varepsilon+s'} { |\hat g(\xi-\xi_*)  |   }         }
{\la\xi_*\ra^{3+\gamma+2s' -3 -2 \varepsilon} } \left(\frac{\la \xi\ra}{\la \xi_* \ra}\right)^{1-s'}
d\xi_*  \Big) d\xi \\
&\lesssim  \|f\|_{H^{s_0-\frac{3}{2}-\varepsilon}}
\|\la D\ra^{s'+ N(\tau-\kappa) -\frac{3}{2} -\varepsilon} g\|_{L^2}\|h\|_{H^{s'}} \,,
\end{align*}
if $\gamma+2s' > 2\varepsilon$.
Putting $f =g = f(\tau, v)$ and $h = M_{\lambda_2(\tau)}^\delta f(\tau, v)$ we have a term coming from $\tilde A_{1,\lambda_2}$
in estimating
\[
\big |\int_\kappa^t
\Big(   M_{\lambda_2}^\delta  Q(f(\tau),  f(\tau))  -Q(f(\tau), M_{\lambda_2}^\delta f(\tau)), M_{\lambda_2}^\delta f(\tau) \Big)  d\tau\big|
\]
as follows:
\begin{align*}
&\Big(\sup_{\tau \in [\kappa,t] }      \|f(\tau)\|_{H^{s_0-\frac{3}{2}-\varepsilon}} \Big)
\Big\{
\int_\kappa^t
\|M_{\lambda_2(\tau)}^\delta f(\tau)\|^2_{H^{s'}} d\tau  \Big \}
\\
&\qquad \qquad +
\int_\kappa^t   \|\la D\ra^{s'+ N(\tau-\kappa) -\frac{3}{2} -\varepsilon} f(\tau)\|^2_{L^2} d\tau\Big\}\,.
\end{align*}
In order to avoid the confusion we write $N =N_2 = s_0/(t-\kappa)$ in this second step and $N = N_1
= s_0/t$ in \eqref{integral-est}. Then we have
\[N_2 (\tau-\kappa) \le N_1 \tau\, \enskip \mbox{if }\enskip \tau \in [\kappa, t]\,,
\]
from which we can use \eqref{integral-est} to estimate the term coming from $\tilde A_{1,\lambda_2}$.
In this step we finally obtain, in view of $N(t-\kappa) = s_0$,
\begin{align*}
\|\la D \ra^{2s_0 -3/2 -\varepsilon} f(t)\|_{L^2} < \infty
\end{align*}
and
\begin{align}\label{integral-est-good}
\int_\kappa^t \|\la D \ra^{s_0 -3/2 -\varepsilon} f(\tau)\|^2_{H^{s'}} d\tau < \infty.
\end{align}

\smallskip
\noindent
{\bf 3rd step:}  For $k \ge 2$, suppose that
\begin{align*}
\sup_{\tau \in [(k-1)\kappa, t] }\|\la D \ra^{(k -1)s_0 -3/2 -\varepsilon} f(\tau)\|_{L^2} < \infty\,.
\end{align*}
For arbitrary $t> k \kappa$ and $N >0$ satisfying $N(t- k \kappa) = s_0$ we set
\[
\lambda_k(\tau) = (k-1)s_0 +N (\tau-\kappa) -\frac 3 2 -\varepsilon \enskip \mbox{for} \enskip \tau \in [\kappa, t]\,.
\]
Consider $M^\delta_{\lambda_k(\tau)}$. Then,
using \eqref{integral-est-good} instead of  \eqref{integral-est}, we can proceed the induction method
by the almost same way as in
the second step.
Since $\kappa >0$ is arbitrary we obtain the desired conclusion.

\smallskip
\noindent
{\bf Acknowledgements :}
The research of the first author was supported in part by the Zhiyuan
foundation and Shanghai Jiao Tong University.
The research of the second author was
supported by  Grant-in-Aid for Scientific Research No.22540187,
Japan Society of the Promotion of Science. The research of the fourth  author was supported partially by `` the Fundamental Research Funds for the Central Universities''. The last author's research was supported by the General Research Fund of Hong Kong, CityU No.103109, and the Lou Jia Shan Scholarship programme of Wuhan University.


\smallskip

\begin{thebibliography}{99}

\bibitem{advw} R. Alexandre, L. Desvillettes, C. Villani, B. Wennberg, \textit{Entropy dissipation and long-range interactions}, Arch. Ration. Mech. Anal. 152 (2000), 327-355
\bibitem{alexandre1} R. Alexandre, M. Safadi, \textit{Littlewood Paley decomposition and regularity issues in Boltzmann homogeneous equations. I. Non cutoff and Maxwell cases}, Math. Models Methods Appl. Sci. 15 (6) (2005), 907-920
\bibitem{al-saf-1}R. Alexandre, M. Safadi, \textit{Littlewood-Paley theory and regularity issues in
Boltzmann homogeneous equations. II. Non cutoff case and non
Maxwellian molecules.}, { Discrete Contin. Dyn. Syst}. {\bf 24} (2009)
1-11.

\bibitem{amuxy3} R. Alexandre, Y. Morimoto, S. Ukai, C.-J. Xu
and T. Yang, \textit{Regularizing effect and local existence
for non-cutoff Boltzmann equation}, { Arch. Rational Mech. Anal.},{\bf 198} (2010), 39-123.

\bibitem{amuxy4-1} R. Alexandre, Y. Morimoto, S. Ukai, C.-J. Xu
and T. Yang, \textit{Boltzmann equation without angular cutoff in the whole
space: I, Global existence for soft potential},
to appear in  { J. Funct. Anal.},  http://hal.archives-ouvertes.fr/hal-00496950/fr/


\bibitem{amuxy4-2} R. Alexandre, Y. Morimoto, S. Ukai, C.-J. Xu
and T. Yang,  \textit{Boltzmann equation without angular cutoff in the whole
space: II, global existence for hard potential},
{Analysis and Applications, {\bf  9-2}(2011), 1-22.}.

\bibitem{amuxy4-3} R. Alexandre, Y. Morimoto, S. Ukai, C.-J. Xu
and T.Yang,
\textit{Boltzmann equation without angular cutoff in the whole space: III, Qualitative properties of solutions},
to appear in {Arch. Rational Mech. Anal.},
http://hal.archives-ouvertes.fr/hal-00510633/fr/.


\bibitem{bobylev-2}
A. Bobylev, \textit {Moment inequalities for the Boltzmann equations and applications to spatially homogeneous problems,}
{J. Statist. Phys.}, {\bf 88}(1997), 1183--1214


\bibitem{dev-mouhot} L. Desvillettes, C. Mouhot, \textit{Stability and uniqueness for the spatially homogeneous Boltzmann equation with long-range interactions}, {Arch. Ration. Mech. Anal.} {\bf 193} (2009), no. 2, 227-253.

\bibitem{desv-wen1} L. Desvillettes, B. Wennberg, \textit{Smoothness of the solution
 of the spatially homogeneous Boltzmann equation without cutoff}.{
 Comm. Partial Differential Equations} {\bf 29-1-2}, 133--155(2004)

\bibitem{H-C} Y. Chen, L. He, \textit{ Smoothing estimates for Boltzmann equation with full-range interactions: Spatially homogeneous case},
{ Arch. Rational Mech. Anal.}{ doi: 10.1007/s00205-010-0393-8}

\bibitem{HMUY} Z.H. Huo, Y. Morimoto, S. Ukai, T. Yang, \textit{Regularity of
solutions for spatially homogeneous Boltzmann equation without
Angular cutoff.} { Kinetic and Related Models}, {\bf 1} (2008)
453-489.

\bibitem{MUXY-DCDS}
Y. Morimoto, S. Ukai, C.-J. Xu, T. Yang, \textit{Regularity of solutions to the
spatially homogeneous Boltzmann equation without angular cutoff}.
{Discrete and Continuous Dynamical Systems -
 Series A} {\bf 24}, 187--212(2009)

\bibitem{ukai}S. Ukai, \textit{Local solutions in Gevrey classes
to the nonlinear Boltzmann equation
without cutoff}.  Japan J. Appl. Math.{\bf 1-1}(1984), 141--156.


\bibitem{villani}C. Villani, \textit{On a new class of weak solutions to the spatially
homogeneous Boltzmann and Landau equations}. Arch. Rational
 Mech. Anal. {\bf 143}, 273--307(1998)

\bibitem{villani2}C. Villani, A review of mathematical
topics in collisional kinetic theory. In: Friedlander S.,
Serre D. (ed.)
Handbook of Fluid Mechanics  (2002).

\bibitem {We96} B. Wennberg,
\textit {The Povzner inequality and moments in the Boltzmann
equation,}
 in ``Proceedings of the VIII International Conference on Waves
and Stability in Continuous Media", Part II (Palermo, 1995);  Rend.
Circ. Mat. Palermo (2) Suppl. No. 45, part II (1996), 673--681.

\end{thebibliography}
\end{document}